\documentclass[12pt, reqno]{amsart}
\usepackage{amssymb,amsthm,amsfonts,amsmath, amscd}

\usepackage{hyperref}
\usepackage{mathrsfs}

\newcommand\Y{\mathbb Y}
\newcommand\Z{\mathbb Z}
\newcommand\C{\mathbb C}
\newcommand\R{\mathbb R}

\newcommand\GT{{\mathbb{GT}}}
\newcommand\Pas{\mathbb{P}}
\newcommand\YB{\mathbb{YB}}

\newcommand\al{\alpha}
\newcommand\be{\beta}
\newcommand\ga{\gamma}
\newcommand\Ga{\Gamma}
\newcommand\de{\delta}

\newcommand\La{\Lambda}
\newcommand\la{\lambda}

\newcommand\om{\omega}
\newcommand\Om{\Omega}

\newcommand\wt{\widetilde}
\newcommand\wh{\widehat}

\newcommand\const{\operatorname{const}}
\newcommand\Dim{\operatorname{Dim}}

\newcommand\Sym{\operatorname{Sym}}

\newcommand\down{{\downarrow}}
\newcommand\pd{\partial}
\newcommand\FS{F\!S}

\newcommand\B{\mathcal B}
\newcommand\TT{\mathcal T}
\newcommand\M{\mathcal M}

\newcommand\LB{{}^{\mathbb B}\!\La}
\newcommand\LY{{}^{\Y}\!\La}
\newcommand\LYB{{}^{\YB}\!\La}
\newcommand\LGT{{}^{\GT}\!\La}
\newcommand\MB{{}^{\mathbb B}\!M}
\newcommand\MY{{}^{\Y}\!M}
\newcommand\MYB{{}^{\YB}\!M}
\newcommand\MGT{{}^{\GT}\!M}

\newcommand\z{{(z,z')}}
\newcommand\zw{{(z,z',w,w')}}

\newtheorem{theorem}{Theorem}[subsection]
\newtheorem{proposition}[theorem] {Proposition}

\newtheorem{corollary}[theorem]{Corollary}
\newtheorem{lemma}[theorem]{ Lemma}

\theoremstyle{definition}
\newtheorem{definition}[theorem]{Definition}
\newtheorem{remark}[theorem]{Remark}
\newtheorem{example}[theorem]{Example}

\numberwithin{equation}{subsection}

\begin{document}

\title[The Young bouquet]
{The Young bouquet and its boundary}

\author{Alexei Borodin}
\address{California Institute of Technology, USA;
\newline\indent Massachusetts Institute of Technology, USA;
\newline\indent Institute for Information Transmission Problems, Moscow, Russia}
\email{borodin@math.mit.edu}

\author{Grigori Olshanski}
\address{Institute for Information Transmission Problems, Moscow, Russia;
\newline\indent Independent University of Moscow, Russia} \email{olsh2007@gmail.com}

\begin{abstract}
The classification results for the extreme characters of two basic ``big''
groups, the infinite symmetric group $S(\infty)$ and the infinite-dimensional
unitary group $U(\infty)$, are remarkably similar. It does not seem to be
possible to explain this phenomenon using a suitable extension of the
Schur-Weyl duality to infinite dimension. We suggest an explanation of a
different nature that does not have analogs in the classical representation
theory.

We start from the combinatorial/probabilistic approach to characters of ``big''
groups initiated by Vershik and Kerov. In this approach, the space of extreme
characters is viewed as a boundary of a certain infinite graph. In the cases of
$S(\infty)$ and $U(\infty)$, those are the Young graph and the Gelfand--Tsetlin
graph, respectively. We introduce a new related object that we call the Young
bouquet. It is a poset with continuous grading whose boundary we define and
compute. We show that this boundary is a cone over the boundary of the Young
graph, and at the same time it is also a degeneration of the boundary of the
Gelfand--Tsetlin graph.

The Young bouquet has an application to constructing infinite-dimensional
Markov processes with determinantal correlation functions.

\end{abstract}

\date{}

\maketitle

\tableofcontents

\section{Introduction}\label{sect1}

We start with a brief historic survey whose goal is to explain the motivation
behind our work. A description of our results starts in Section \ref{sc:1.5}.

\subsection{Characters of $S(n)$ and $U(N)$}

The symmetric group $S(n)$ of permutations of an $n$-element set
 is a simple yet fundamental example of a noncommutative finite group.
Similarly, the unitary group $U(N)$ of complex unitary matrices of size $N$ is
a basic example of a noncommutative compact group.

As is well known, the representation theory began with a sequence of papers by
Frobenius that culminated in a masterful computation of the irreducible
characters of $S(n)$ (see e.~g. Curtis \cite{Cur99} and references therein). An
analogous result for $U(N)$ was obtained by Weyl (see \cite{Wey39} and
references therein to Weyl's earlier journal publications of the twenties).

In modern textbooks one can find different approaches to those results, but if
one compares the original arguments of Frobenius and Weyl then their similarity
is apparent. In essence, Weyl builds his approach following Frobenuis' path.

Furthermore, the famous Schur-Weyl duality establishes a direct link between
the characters from the two families. With this duality and relatively simple
additional arguments, one can derive Weyl's character formula from the formula
of Frobenius and \emph{vice versa}. One reason for that is that the characters
of $S(n)$ and $U(N)$ have a common combinatorial base --- the Schur symmetric
functions.

Of course if one views the unitary groups $U(N)$ as a special case of the
reductive Lie groups and constructs a general theory of finite-dimensional
representations of those following the infinitesimal approach (replacing groups
by their Lie algebras) and Cartan's theory of highest weight, then the analogy
with representations of symmetric groups becomes more vague.

However, one can look at a different aspect of the theory --- explicit matrix
realization of representations. There are two classical results here, Young's
orthogonal form for the irreducible representations of $S(n)$ and
Gelfand-Tsetlin's formulas for the irreducible representations of $U(N)$. Both
results are based on the existence of a basis in an irreducible representation
that is connected to a chain of subgroups
\begin{equation}\label{eq1.A}
S(1)\subset S(2)\subset\dots\subset S(n) \quad   \textrm{and} \quad U(1)\subset
U(2)\subset\dots\subset U(N),
\end{equation}
respectively, and the analogy between the realizations in the Young basis and
in the Gelfand-Tsetlin basis is very clear (some authors even use the term
``Gelfand-Tsetlin basis'' for Young's basis as well).

Thus, one observes relations between symmetric and unitary groups both on the
level of characters and on the level of matrix realizations of the irreducible
representations. This is surprising as the groups themselves are structurally
quite different.

\subsection{Characters of $S(\infty)$ and $U(\infty)$}

One can go even further. Let us extend the group chains \eqref{eq1.A} to
infinity and consider the corresponding inductive limits --- the infinite
symmetric group $S(\infty):=\varinjlim S(n)$ and the infinite-dimensional
unitary group $U(\infty):=\varinjlim U(N)$. These two groups are neither finite
nor compact, and $U(\infty)$ is not even locally compact. Nevertheless, one can
modify the definition of an irreducible character in such a way that it would
make perfect sense for such ``big'' groups. We have in mind the so-called
extreme (or indecomposable) characters that correspond to finite factor
representations in the sense of von Neumann. (For the finite and compact groups
the extreme characters differ from the conventional irreducible ones only by
normalization.)

The extreme characters of $S(\infty)$ were first considered by Thoma
\cite{Tho64}, and 12 years later Voiculescu \cite{Vo76} wrote a paper on the
extreme characters of $U(\infty)$. It was discovered later (Vershik and Kerov
\cite{VK81}, \cite{VK82}; Boyer \cite{Boy83}) that the classification of the
extreme characters of both groups was implicitly contained in earlier works of
Schoenberg and his followers on totally positive matrices (Aissen, Edrei,
Schoenberg, and Whitney \cite{AESW51}; Aissen, Schoenberg, and Whitney
\cite{ASW52}; Edrei \cite{Ed52}, \cite{Ed53}).

It turns out that on the level of inductive limits the analogy between the
symmetric and unitary groups becomes even more apparent. The character formulas
of Thoma and Voiculescu are remarkably similar, and in the language of total
positivity the character classification admits a uniform description: In both
cases there exists a bijective correspondence between the extreme characters
and infinite totally positive Toeplitz matrices; in the first case (for
$S(\infty)$) one needs to consider only triangular matrices while in the second
case (for $U(\infty)$) no restriction is necessary. In both cases the
characters depend on infinitely many continuous parameters, and the set of
parameters for $U(\infty)$ is roughly double of that for $S(\infty)$.

\subsection{Harmonic analysis on $S(\infty)$ and $U(\infty)$}

The term ``harmonic analysis'' (in noncommutative setting) usually refers to
the set of questions related to the decomposition of the regular representation
and its relatives on irreducibles. However, for inductive limits like
$S(\infty)$ or $U(\infty)$, questions of that sort seemingly do not make sense.
For example, the group $U(\infty)$ does not have a Haar measure so its regular
representation simply does not exist. Nevertheless, there is a way of
circumvent this obstacle and construct a whole family of representations each
of which could play the role of the regular one.

The original idea is due to Pickrell \cite{Pic87}, Neretin presented its
generalization in
 \cite{Ner02}, and further developments (detailed analysis of the representations)
can be followed along Borodin and Olshanski \cite{BO00a}, \cite {BO05a},
\cite{BO05b}; Gorin \cite{Gor10}; Kerov, Olshanski, and Vershik \cite{KOV93},
\cite{KOV04}; Olshanski \cite{Ols03b}, \cite{Ols03c}; Osinenko \cite{Osi11}.
Some of these articles deal with the unitary group while the other ones deal
with the symmetric group, and once again one easily sees the parallelism
between the two cases. It shows in constructing extensions of the groups that
allow to define analogs of the Haar measure, in defining analogs of the regular
representation, and in the structure of decomposition of those.

\subsection{The Young graph and the Gelfand-Tsetlin graph}

The set of extreme characters of a given group $G$ may be viewed as a variant
of the dual object to $G$; for that reason we use the notation $\wh G$. Vershik
and Kerov (\cite{VK81}, \cite{VK90}) were first to observe that the dual object
$\wh{S(\infty)}$ to the infinite symmetric group can be defined in purely
combinatorial/probabilistic terms. More exactly, $\wh{S(\infty)}$ serves as a
``boundary'' for an infinite graph called the Young graph. Similarly,
$\wh{U(\infty)}$ is the ``boundary'' of a different graph called the
Gelfand--Tsetlin graph. (The term ``boundary'' carries roughly the same meaning
as in the theory of Markov processes; an exact definition is given in Section
\ref{sect2.2}.)

This interpretation leads to a fruitful connection between noncommutative
harmonic analysis and probability theory: As shown in \cite{BO09} and
\cite{BO10}, the spectral measures on the dual objects $\wh{S(\infty)}$ and
$\wh{U(\infty)}$ that arise from decomposing regular representations, serve as
stationary distributions for certain Markov processes.

The Young graph, denoted as $\Y$, encodes branching of the irreducible
characters of the group chain
\begin{equation*}
S(1)\subset S(2)\subset\dots\subset S(n)\subset S(n+1)\subset\cdots
\end{equation*}
Namely, the set of vertices of $\Y$ is the disjoint union of the dual objects
\begin{equation*}
\wh{S(1)}\sqcup \wh{S(2)} \sqcup \dots\sqcup \wh{S(n)} \sqcup \wh{S(n+1)}\sqcup
\cdots.
\end{equation*}
Since the irreducible characters of $S(n)$ are parametrized by the Young
diagrams with $n$ boxes, the set of vertices can be identified with the set of
all Young diagrams. Further, two vertices are joined by an edge if the
corresponding diagrams are different by exactly one box. This definition
reflects Young's branching rule: The restriction of the irreducible character
of $S(n+1)$ indexed by a Young diagram $\nu$ to $S(n)$ is the sum of exactly
those characters whose diagrams are obtained from $\nu$ by deleting a single
box.

Similarly, the Gelfand-Tsetlin graph, denoted as $\GT$, encodes branching of
the irreducible characters for the chain
\begin{equation*}
U(1)\subset U(2)\subset\dots\subset U(N)\subset U(N+1)\subset\cdots.
\end{equation*}
The set of vertices in $\GT$ is the disjoint union
\begin{equation*}
\wh{U(1)}\sqcup \wh{U(2)} \sqcup \dots\sqcup \wh{U(N)} \sqcup \wh{U(N+1)}\sqcup
\cdots.
\end{equation*}
The irreducible characters of $U(N)$ are parametrized by the integer-valued
vectors of length $N$ with nonincreasing coordinates,
$$
\mu=(\mu_1\ge\dots\ge\mu_N)\in\Z^N.
$$
Such vectors are called signatures. According to the branching rule for
irreducible characters of the unitary groups, two signatures of length $N$ and
$N+1$ are connected by an edge if their coordinates interlace:
$$
\la_1\ge \mu_1\ge\la_2\ge\dots \ge \la_N\ge \mu_N\ge\la_{N+1}.
$$

Both graphs $\Y$ and $\GT$ are graded in such a way that the edges can only
join vertices of adjacent levels. In $\Y$, the vertices of the level
$n=1,2,\dots$ are those Young diagrams that have exactly $n$ boxes, while in
$\GT$ the vertices of level $N=1,2,\dots$ are the signatures of length exactly
$N$.

Observe that any signature $\la=(\la_1,\dots,\la_N)$ can be viewed as a pair of
Young diagrams $(\la^+,\la^-)$, where the nonzero lengths of rows in $\la^+$
are the positive coordinates in $\la$, and the nonzero lengths of rows in
$\la^-$ are the absolute values of the negative coordinates in $\la$. This
observation contains a hint at the above mentioned fact that $\wh{U(\infty)}$
(= the boundary of $\GT$) has doubly many parameters comparing to
$\wh{S(\infty)}$ (= the boundary of $\Y$).

We now proceed to the content of the present article.

\subsection{What is the Young bouquet}\label{sc:1.5}

Although our last comment points to a certain similarity between $\Y$ and
$\GT$, the grading of the two is totally different: $n$ is the number of boxes
of a diagram (equivalently, the sum of lengths of its rows), while $N$ is the
length of a signature (or the number of its coordinates). Even if all the
coordinates of a signature $\la$ are nonnegative, i.~e. in the correspondence
$\la=(\la^+,\la^-)$ the second diagram $\la^-$ is empty and $\la$ is seemingly
reduced to $\la^+$, the quantities $n$ and $N$ have very different meanings.

The main idea of this paper is that in order to see a clear connection between
the graphs $\Y$ and $\GT$, one needs to introduce an intermediate object. This
new object, that we call the Young bouquet and denote as $\YB$, is not a graph.
However, $\YB$ is a graded poset, similarly to $\Y$ and $\GT$. One new feature
is that the grading in $\YB$ is not discrete but continuous; the grading level
is marked by a positive real number. By definition, the elements of $\YB$ of a
given level $r>0$ are pairs $(\nu, r)$, where $\nu$ is an arbitrary Young
diagram. The partial order in $\YB$ is defined as follows: $(\nu,r)<(\wt\nu,\wt
r)$ if $r<\wt r$ and diagram $\nu$ is contained in diagram $\wt\nu$ (or
coincides with it).

We explain how the boundary of the Young bouquet should be understood, and show
(Theorem \ref{thm4}) that it is a cone over the boundary of the Young graph.
This establishes a connection between $\Y$ and $\YB$. We also note that the
partial order in $\YB$ is obviously consistent with the inclusion partial order
on $\Y$.

On the other hand, we show that $\YB$ can be obtained from $\GT$ by a
degeneration procedure that can also be viewed as a kind of scaling limit
transition. More exactly, one has to start with $\GT$'s subgraph $\GT^+$
consisting of signatures with nonnegative coordinates, and in the degeneration
$\GT^+\to\YB$ one renormalizes the levels, which turns the discrete grading
into a continuous one.

Because of these two relationships, with $\Y$ and with $\GT$, we say that $\YB$
is a suitable intermediate object between $\Y$ and $\GT$.

The notion of Young bouquet is perfectly consistent with the concept of ``grand
canonical ensembles'' of random Young diagrams: The well-known model of
poissonized Plancherel measures \cite{BDJ99} and a more general model of mixed
z-measures \cite{BO00a} become more natural when placed within the context of
the Young bouquet.

\subsection{Degeneration $\GT^+\to\YB$}

While the connection between $\Y$ and $\YB$ is fairly obvious, the degeneration
$\GT^+\to\YB$ deserves to be explained in more detail.

(a) An exact statement of what we mean by the degeneration of the graph $\GT^+$
to the poset $\YB$ is contained in Theorem \ref{thm5}. The statement involves a
degeneration of a certain transition function that is canonically associated to
$\GT$, to the transition function canonically associated to $\YB$. (Let us also
mention here that our ``boundary'' is always the entrance boundary for a
certain transition function. The graph and poset structure are mostly needed to
define that transition function.)

(b) In Theorem \ref{thm7} we explain in what sense the boundary of $\YB$
(recall that it is a cone over the boundary of $\Y$) can be obtained as a
degeneration of the boundary of $\GT^+$.

(c) Theorem \ref{thm6} shows that the degeneration $\GT^+\to\YB$ is accompanied
by degeneration of certain probability measures that originate in harmonic
analysis on $S(\infty)$ and $U(\infty)$. This aspect of the degeneration
$\GT^+\to\YB$ can be compared to a descent in the hierarchy of the
hypergeometric orthogonal polynomials.

(d) Finally, in Section \ref{sect5} we discuss the spaces of monotone paths in
the posets $\Y$, $\YB$, $\GT$, and Gibbs measures on those spaces. We show that
the degeneration $\GT^+\to\YB$ can be described in this context as well. The
finite monotone paths in $\Y$ and $\GT^+$ have well known combinatorial
interpretations; these are the standard and semistandard Young tableaux,
respectively. One can interpret the finite monotone paths in $\YB$ in a similar
fashion: Those are Young diagrams filled with positive real numbers with the
same monotonicity conditions along rows and columns as in the definition of the
standard Young tableaux.

\subsection{An application}

In \cite{BO10} we constructed a family of Markov processes on the dual object
$\wh{U(\infty)}$ using its identification with the boundary of $\GT$. On the
other hand, \cite{Ols10} contained an announcement of the existence of a
similar model of Markov dynamics, where the state space is the cone over
$\wh{S(\infty)}$; in another interpretation, this is a dynamical model of
determinantal processes with infinitely many particles. The construction of the
Young bouquet allows one to give a simpler proof of that result of \cite{Ols10}
using the approach of \cite{BO10}; this is a subject of the follow-up paper
\cite{BO11b}.

\subsection{Acknowledgments}

A.~B. was partially supported by NSF-grant DMS-1056390. G.~O. was partially
supported by a grant from Simons Foundation (Simons--IUM Fellowship), the
RFBR-CNRS grant 10-01-93114, and the project SFB 701 of Bielefeld University.

\section{Graded graphs and projective systems}\label{sect2}

\subsection{The category $\B$}\label{sect2.1}
About the notions used in this subsection see \cite{Mack57} and \cite{Mey66}. A
{\it measurable space\/} (also called {\it Borel space\/}) is a set with a
distinguished sigma-algebra of subsets. Denote by $\B$ the category whose
objects are {\it standard\/} measurable spaces and morphisms are Markov
kernels. A morphism between two objects will be denoted by a dash arrow,
$X\dasharrow Y$, to emphasize that it is not an ordinary map. Recall that a
(stochastic) {\it Markov kernel\/} $\La:X\dasharrow Y$ between two measurable
spaces $X$ and $Y$ is a function $\La(a,A)$, where $a$ ranges over $X$ and $A$
ranges over measurable subsets of $Y$, such that $\La(a,\,\cdot\,)$ is a
probability measure on $Y$ for any fixed $a$ and $\La(\,\cdot\,,A)$ is a
measurable function on $X$ for any fixed $A$.

Below we use the short term {\it link\/} as a synonym of ``Markov kernel''. The
composition of two links will be read from left to right: Given
$\La:X\dasharrow Y$ and $\La':Y\dasharrow Z$, their composition
$\La\La':X\dasharrow Z$ is defined as
$$
(\La\La')(x,dz)=\int_Y \La(x,dy)\La'(y,dz),
$$
where $\La(x,dy)$ and $\La'(y,dz)$ symbolize the measures $\La(x,\,\cdot\,)$
and $\La'(y,\,\cdot\,)$, respectively.

A {\it projective system\/} in $\B$ is a family $\{V_i, \La^j_i\}$ consisting
of objects $V_i$ indexed by elements of a linearly ordered set $I$ (not
necessarily discrete), together with links $\La^j_i:V_j\dasharrow V_i$ defined
for any couple $i<j$ of indices, such that for any triple $i<j<k$ of indices,
one has $ \La^k_j\La^j_i=\La^k_i$.

A {\it limit object\/}  of a projective system is understood in the categorical
sense: This is an object $X=\varprojlim V_i$ together with links
$\La^\infty_i:X\dasharrow V_i$ defined for all $i\in I$, such that:

\begin{itemize}

\item[$\bullet$] $\La^\infty_j\La^j_i=\La^\infty_i$ for all $i<j$;

\item[$\bullet$] if an object $Y$ and links $\wt\La^\infty_i:Y\dasharrow V_i$
satisfy the similar condition, then there exists a unique link
$\La^Y_X:Y\dasharrow X$ such that $\wt\La_i=\La^Y_X\La^\infty_i$.

\end{itemize}

General results concerning existence and uniqueness of limit objects in $\B$
can be found in Winkler \cite[Chapter 4]{Wi85}. See also Dynkin \cite{Dy71},
\cite{Dy78}, Kerov and Orevkova \cite{KeOr90}. When the index set $I$ is a
subset of $\R$ and all spaces $V_i$ are copies of one and the same space $X$,
our definition of projective system turns into the classical notion of {\it
transition function\/} on $X$ (within inversion of order on $I$).

For a measurable space $X$ we denote by $\M(X)$ the set of probability measures
on $X$. It is itself a measurable space: the corresponding sigma-algebra is
generated by the sets of the form $\{\mu\in\M(X): \mu(A)\in B\}$, where
$A\subseteq X$ is a measurable  and $B\subseteq\R$ is Borel. Equivalently, the
measurable structure of $\M(X)$ is determined by the requirement that for any
bounded measurable function on $X$, its coupling with $M$ should be a
measurable function in $M$. If $X$ is standard, then $\M(X)$ is standard, too.

Observe that any link $\La:X\dasharrow Y$ gives rise to a measurable map
$\M(X)\to \M(Y)$, which we write as $M\mapsto M\La$. Consequently, any
projective system $\{V_i, \La^j_i\}$ in $\B$ gives rise to the conventional
projective limit of sets
$$
\M_\infty:=\varprojlim_I\M(V_i).
$$
An element of $\M_\infty$ is called a {\it coherent family of measures\/}: By
the very definition, it is a family of probability measures $\{M_i\in\M(V_i):
i\in I\}$ such that for any couple $i<j$ one has $M_j\La^j_i=M_i$. (In the case
of a transition function, Dynkin \cite{Dy78} terms elements of $\M_\infty$ {\it
entrance laws\/}.)

If a limit object $X$ exists, then there is a canonical map
$$
\M(X)\to\M_\infty.
$$

From now on we will gradually narrow the setting of the formalism and will
finally focus on the study of some concrete examples.

\subsection{Projective chains}\label{sect2.2}
Consider a particular case of a projective system, where all spaces are
discrete (finite or countably infinite) and the indices range over the set
$\{1,2,\dots\}$ of natural numbers. Such a system is uniquely determined by the
links $\La^{N+1}_N$, $N=1,2,\dots$:
\begin{equation}\label{eq3}
V_1\dashleftarrow V_2\dashleftarrow\dots\dashleftarrow V_N\dashleftarrow
V_{N+1} \dashleftarrow \cdots.
\end{equation}
Note that a link between two discrete spaces is simply a stochastic matrix, so
that $\La^{N+1}_N:V_{N+1}\dasharrow V_N$ is a stochastic matrix whose rows are
parametrized by points of $V_{N+1}$ and columns are parametrized by points of
$V_N$:
\begin{gather*}
\La^{N+1}_N=[\La^{N+1}_N(x,y)], \quad x\in V_{N+1}, \, y\in V_N, \\
\textrm{$\La^{N+1}_N(x,y)\ge0$\; for every  $x,y$, \quad $\sum_{y\in
V_N}\La^{N+1}_N(x,y)=1$ for every $x$}.
\end{gather*}
For arbitrary $N'>N$, the corresponding link $\La^{N'}_N: V_{N'}\dasharrow V_N$
is a stochastic matrix of format $V_{N'}\times V_N$, which factorizes into a
product of stochastic matrices corresponding to couples of adjacent indices:
$$
\La^{N'}_N=\La^{N'}_{N'-1}\dots\La^{N+1}_N.
$$

We call such a projective system a {\it projective chain\/}. It gives rise to a
chain of ordinary maps
\begin{equation}\label{eq3.1}
\M(V_1)\leftarrow
\M(V_2)\leftarrow\dots\leftarrow\M(V_N)\leftarrow\M(V_{N+1})\leftarrow \cdots
\end{equation}
Note that $\M(V_N)$ is a simplex whose vertices can be identified with the
points of $V_N$, and the arrows are affine maps of simplices. In this situation
a coherent family (that is, an element of $\M_\infty$)  is a sequence
$\{M_N\in\M(V_N): N=1,2,\dots\}$ such that
$$
M_{N+1}\La^{N+1}_N=M_N, \qquad N=1,2,\dots\,.
$$
Here we can interpret measures as row vectors, so that the left-hand side is
the product of a row vector by a matrix. In more detail, the equation can be
written as
$$
\sum_{x\in V_{N+1}}M_{N+1}(x)\La^{N+1}_N(x,y)=M_N(y), \qquad \forall y\in V_N.
$$

Note that the set $\M_\infty$ may be empty, as the following simple example
shows: Take $V_N=\{N,N+1,N+2,\dots\}$ and define $\La^{N+1}_N$ as the natural
embedding $V_{N+1}\subset V_N$. In what follows we tacitly assume that
$\M_\infty$ is nonempty. This holds automatically if all $V_N$ are finite sets.

We may view $\M_\infty$ as a subset of the real vector space
$$
L:=\R^{V_1\sqcup V_2\sqcup V_3\sqcup\dots}.
$$
Here the set $V_1\sqcup V_2\sqcup V_3\sqcup\dots$ is the disjoint union of
$V_N$'s. Since this set is countable, the space $L$ equipped  with the product
topology is locally convex and metrizable. Clearly, $\M_\infty$ is a convex
Borel subset of $L$, hence a standard Borel space.

Let $V_\infty$ be the set of extreme points of $\M_\infty$. We call $V_\infty$
the {\it boundary\/} of the chain $\{V_N,\La^{N+1}_N\}$.

\begin{theorem}\label{thm1}
If\/ $\M_\infty$ is nonempty, then the boundary\/ $V_\infty\subset\M_\infty$ is
a nonempty measurable subset {\rm(}actually, a subset of type $G_\de${\rm)}
of\/ $\M_\infty$, and there is a natural bijection
$\M_\infty\leftrightarrow\M(V_\infty)$, which is an isomorphism of measurable
spaces.
\end{theorem}

A proof based on Choquet's theorem is given in \cite[\S9]{Ols03c}, a much more
general result is contained in \cite[Chapter 4]{Wi85}.

By the very definition of the boundary $V_\infty$, it comes with canonical
links
$$
\La^\infty_N:V_\infty\dasharrow V_N, \qquad N=1,2,\dots\,.
$$
Namely, given a point $\om\in V_\infty\subset\M_\infty$, let $\{M_N\}$ stand
for the corresponding sequence of measures; then, by definition,
$$
\La^\infty_N(\om,x)=M_N(x), \qquad x\in V_N, \quad N=1,2,\dots\,.
$$
Here, to simplify the notation, we write $\La^\infty_N(\om,x)$ instead of
$\La^\infty_N(\om,\{x\})$.

{}From the definition of $\La^\infty_N$ it follows that
$$
\La^\infty_{N+1}\La^{N+1}_N=\La^\infty_N, \qquad N=1,2,\dots\,.
$$
Now it is easy to see that the boundary $V_\infty$ coincides with the
categorical projective limit of the initial chain \eqref{eq3}.

\begin{remark}\label{rem2}
In the context of Theorem \ref{thm1}, assume we are given a standard measurable
space $X$ and links $\La^X_N: X\dasharrow V_N$, $N=1,2,\dots$, such that:
\begin{itemize}
\item[$\bullet$] $\La^X_{N+1}\La^{N+1}_N=\La^X_N$ for all $N$;

\item[$\bullet$] the induced map $\M(X)\to\M_\infty=\varprojlim\M(V_N)$ is a
bijection.
\end{itemize}
Then $X$ coincides with the boundary $V_\infty$. Indeed, the maps
$\M(X)\to\M_N$ are measurable, whence the map $\M(X)\to\M_\infty$ is
measurable, too. Since $\M(X)$ is standard (because $X$ is standard), the
latter map is an isomorphism of measurable spaces (see \cite[Theorem
3.2]{Mack57}) and the claim becomes obvious.
\end{remark}

\begin{remark}\label{rem1}
Theorem \ref{thm1} immediately extends to the case of a projective system
$\{V_i, \La^j_i\}$, where all $V_i$'s are discrete spaces (finite or countable)
and the directed index set $I$ is countably generated, that is, contains a
sequence $i(1)<i(2)<\dots$ such that any $i\in I$ is majorated by indices
$i(N)$ with $N$ large enough. Indeed, it suffices to observe that the space
$\varprojlim \M(V_{i(N)})$ does not depend on the choice of $\{i(N)\}$. Such a
situation is examined in Section \ref{sect3}, where the index set $I$ is the
halfline $\R_{>0}$.
\end{remark}

\subsection{Graded and branching graphs}

\begin{definition}\label{def2}
By a {\it graded graph\/} we mean a graph $\Ga$ with countably many vertices
partitioned into {\it levels\/} enumerated by numbers $1,2,\dots$, and such
that (below $|v|$ denotes the level of a vertex $v$):

\begin{itemize}

\item[$\bullet$] if two vertices $v$, $v'$ are joined by an edge then
$|v|-|v'|=\pm1$;

\item[$\bullet$] multiple edges between $v$ and $v'$ are allowed;

\item[$\bullet$] each vertex $v$ is joined with a least one vertex of level
$|v|+1$;

\item[$\bullet$] if $|v|\ge2$, then the set of vertices of level $|v|-1$ joined
with $v$ is finite and nonempty.

\end{itemize}
\end{definition}

This is a natural extension of the well-known notion of a {\it Bratteli
diagram\/} \cite{Br72}: the difference between the two notions is that a
Bratteli diagram has finitely many vertices at each level, whereas our
definition allows countable levels.

Sometimes it is convenient to slightly modify the above definition by adding to
$\Ga$ a single vertex of level 0 joined by edges with all vertices of level 1.

\begin{example}
The simplest nontrivial example of a graded graph is the {\it Pascal graph\/}
$\Pas$, also called the {\it Pascal triangle\/}. The vertices of $\Pas$ are
points $(n_1,n_2)$ of the lattice $\Z^2$ with nonnegative coordinates, the
edges join points with one of the coordinates shifted by $\pm1$, and the level
is defined as the sum $|(n_1,n_2)|=n_1+n_2$. A number of other examples are can
be found in Kerov's book \cite{Ke03} and also in Gnedin \cite{Gn97}, Gnedin and
Olshanski \cite{GO06}, Kingman \cite{Ki78}.
\end{example}

\begin{definition}[Branching graphs]\label{ex2}
Given a chain of finite or compact groups embedded to each other,
\begin{equation}\label{eq32}
G(1)\subset G(2)\subset\dots\subset G(N-1)\subset G(N)\subset\dots,
\end{equation}
one constructs a graded graph $\Ga=\Ga(\{G(N)\})$, called the {\it branching
graph\/} of the group chain \eqref{eq32}, as follows. The vertices of level $N$
are the labels of the equivalence classes of irreducible representations of
$G(N)$. Choose a representation $\pi_v$ for each vertex $v$. Two vertices $u$
and $v$ of levels $N$ and $N-1$, respectively, are joined by $m$ edges if
$\pi_{u}$ enters the the decomposition of $\pi_v\downarrow G(N-1)$ with
multiplicity $m$, with the understanding that that there are no edges if $m=0$.
\end{definition}

Of particular importance for us are two branching graphs: the Young graph and
the Gelfand-Tsetlin graph; they are obtained from the chains of symmetric
groups and compact unitary groups, respectively. These graphs are discussed
below, see Sections \ref{Young} and \ref{GT}.

\begin{definition}
Given a graded graph $\Ga$, the {\it dimension\/} of a vertex $v$, denoted by
$\dim v$, is defined as the number of all (monotone) paths in $\Ga$ of length
$|v|-1$ starting at some vertex of level 1 and ending at $v$ (for more detail
about paths, see Section \ref{sect5.1} below). Further, for an arbitrary vertex
$u$ with $|u|<|v|$, the {\it relative dimension\/} $\dim(u,v)$ is the number of
(monotone) paths of length $|v|-|u|$ joining $u$ to $v$. In particular, if
$|u|=|v|-1$, then $\dim(u,v)$ is the number of edges between $u$ and $v$.
\end{definition}

For instance, in the case of the Pascal graph $\Ga=\Pas$, if $v=(n_1,n_2)$ and
$u=(m_1,m_2)$, $u\ne v$, then the dimensions are binomial coefficients:
$$
\dim v=\frac{(n_1+n_2)!}{n_1!n_2!}, \qquad \dim(u,v)=\begin{cases}
\dfrac{(n_1+n_2-m_1-m_2)!}{(n_1-m_1)!(n_2-m_2)!}, &\textrm{$n_1\ge m_1$ and
$n_2\ge m_2$}\\ 0, &\textrm{otherwise}.
\end{cases}
$$

Note that if $\Ga$ is a branching graph, then $\dim v$ is the dimension of the
corresponding representation $\pi_v$ and $\dim(u,v)$ is the multiplicity of
$\pi_u$ in the decomposition of representation $\pi_v$ restricted to the
subgroup $G(|u|)\subset G(|v|)$.

Obviously, one has
$$
\dim v=\sum_{u:\, |u|=|v|-1} \dim u\,\dim(u,v).
$$
This leads to

\begin{definition}[Projective chains associated to graded graphs]
Any graded graph $\Ga$ gives rise to a chain $\{V_N,\La^{N+1}_N\}$, where $V_N$
consists of the vertices of level $N$ and
$$
\La^{N+1}_N(v,u)=\frac{\dim u\cdot\dim(u,v)}{\dim v}, \qquad v\in V_{N+1},
\quad u\in V_{N}.
$$
The boundary $V_\infty$ of this chain is also referred to as the {\it boundary
of the graph\/} $\Ga$ and denoted as $\pd\Ga$.
\end{definition}

More generally, for  $N<N'$ we set
$$
\La^{N'}_N:=\La^{N'}_{N'-1}\dots\La^{N+1}_N.
$$
Then
\begin{equation}\label{eq22}
\La^{N'}_N(v,u)=\frac{\dim u\cdot\dim(u,v)}{\dim v}, \qquad u\in V_N, \quad
v\in V_{N'}.
\end{equation}

If $\Ga$ is a branching graph coming from a group chain \eqref{eq32}, then the
boundary $\pd\Ga$ has a representation-theoretic meaning. Namely, the points of
$\pd\Ga$ can be identified with the indecomposable normalized characters of the
inductive limit group $G(\infty):=\varinjlim G(N)$ (these are the normalized
traces of finite factor representations of $G(\infty)$). See Thoma
\cite{Tho64}, Vershik and Kerov \cite{VK90}, Voiculescu \cite{Vo76}.

\begin{example}[The boundary of the Pascal graph $\Pas$]
The boundary $\pd\Pas$ can be identified with the closed unit interval
$[0,1]\subset\R$ (this fact is equivalent to de Finetti's theorem, see Section
\ref{sect5.2} below). For $\om\in[0,1]$ and a vertex $v=(n_1,n_2)$ of level
$N=n_1+n_2$ one has
$$
\La^\infty_N(\om,v)=\frac{(n_1+n_2)!}{n_1!n_2!}\om^{n_1}(1-\om)^{n_2}.
$$
Thus $\La^\infty_N(\om,\,\cdot\,)$ is the binomial distribution on
$\{0,\dots,N\}$ with parameter $\om$. Note also that
$$
\La^N_{N-1}(v,v')=\begin{cases}\dfrac{n_1}{n_1+n_2}, & v'=(n_1-1,n_2)\\
\smallskip\\
\dfrac{n_2}{n_1+n_2}, & v'=(n_1,n_2-1).
\end{cases}
$$
\end{example}

\section{The Young bouquet}\label{sect3}

\subsection{The binomial projective system $\mathbb B$}
Here we discuss a simple example of a projective system with continuous index
set. This system will serve us as a building block in a more complex
construction.

\begin{definition}
The {\it binomial projective system\/} $\mathbb B$ has the index set
$I=\R_{>0}$ (strictly positive real numbers). All the spaces $V_r$ are discrete
and are copies of the set $\Z_+:=\{0,1,2,\dots\}$ of nonnegative integers. The
links are defined by formula
\begin{equation}\label{eq11}
\LB^{r'}_{r}(n, m)=
\left(1-\frac{r}{r'}\right)^{n-m}\left(\dfrac{r}{r'}\right)^m\,
\dfrac{n!}{(n-m)!\,m!}\,, \qquad n,\, m\in\Z_+.
\end{equation}
\end{definition}

\medskip

Note that the right-hand side vanishes unless $m\le n$. For $n$ fixed the
quantities $\LB^{r'}_{r}(n, m)$ form the binomial distribution on
$\{0,1,\dots,n\}$ with parameter $r/r'$, which explains the name of the system.

Clearly, $\LB^{r'}_{r}$ is a stochastic matrix. Thus, to see that the
definition is correct we have only check the compatibility condition
$$
\LB^{r''}_{r'}\, \LB^{r'}_{r}=\LB^{r''}_{r}, \qquad r''> r'>r.
$$
Or, in more detail,
$$
\sum_{n}\LB^{r''}_{r'}(l,n)\,\LB^{r'}_{r}(n,m)=\LB^{r''}_{r}(l,m).
$$
But this is an easy exercise.

\begin{remark}
Setting $r=e^{-t}$ we may view the binomial projective system as a
time-stationary transition function on $\Z_+$:
$$
p(s,n;t,m)=\left(1-e^{s-t}\right)^{n-m}e^{(s-t)m}\, \dfrac{n!}{(n-m)!\,m!}\,,
\qquad s<t, \quad  n,\, m\in\Z_+.
$$
\end{remark}

\medskip

By virtue of Remark \ref{rem1} we may speak about the boundary $\pd\mathbb B$
of the binomial system. This boundary is described in the following theorem:

\begin{theorem}\label{thm2}
The boundary of the binomial projective system $\mathbb B$ is the space\/
$\R_+:=\{x\in\R: x\ge0\}$ with the links $\LB^\infty_r:\R_+\dasharrow\Z_+$
defined by the Poisson distributions
$$
\LB^\infty_r(x, m)=e^{-r x}\frac{(rx)^m}{m!}\,, \qquad x\in\R_+, \quad
m\in\Z_+.
$$
\end{theorem}

Before proceeding to the proof of the theorem we will prove two simple lemmas.

\begin{lemma}\label{lemma5}
Let $r>0$ and $k\in\Z_+$ be fixed. For any $r'>r$, the function
$$
x\mapsto \left(1-\frac r{r'}\right)^{r'x}x^k
$$
belongs to the Banach space $C_0(\R_+)$ of continuous functions on $\R_+$
vanishing at infinity, with the supremum norm. In the limit as parameter $r'$
goes to $+\infty$, this function converges in the metric of $C_0(\R_+)$ to the
function
$$
x\mapsto e^{-rx}x^k.
$$
\end{lemma}

\begin{proof}
Clearly, the convergence holds uniformly on $x$ in any bounded interval $[0,a]$. On the
other hand, it is easy to estimate the tail of the pre-limit function for $x$
near infinity: As $x\to+\infty$, the function tends to $0$ uniformly on $r'\gg
r$, because
$$
\left(1-\frac r{r'}\right)^{r'}=e^{-r}\left(1+O(1/r')\right), \qquad
\textrm{$r'$ large.}
$$
This proves the lemma.
\end{proof}

\begin{lemma}\label{lemma6}
For any $r>0$, the map $M\mapsto M_r:=M\,\LB^\infty_r$ from $\M(\R_+)$ to
$\M(\Z_+)$ is injective.
\end{lemma}

\begin{proof}
Indeed, given $M\in\M(\R_+)$, its image $M_r$ under $\LB^\infty_r$ is given by
$$
M_r(m)=\frac1{m!}\int_{\R_+}M(dx)e^{-rx}(rx)^m, \qquad m\in\Z_+.
$$
The trivial estimate
$$
e^{-rx}\frac{(rx)^m}{m!}\le1, \qquad x\in\R_+,
$$
entails $e^{-rx}x^m\le m!r^{-m}$. Since $M$ is a probability measure, this
implies that the $m$th moment of measure $M(dx)e^{-rx}$ does not exceed
$m!r^{-m}$. It follows that the exponential generating function for the moments
is analytic in the open disc of radius $r$, which guarantees that the
corresponding moment problem is definite. Therefore, the initial measure
$M(dx)e^{-rx}$ is recovered from its moments uniquely, so that $M$ is
determined by $M_r$ uniquely.
\end{proof}

The following corollary will be used in \cite{BO11b}.

\begin{corollary}\label{cor3}
For any fixed $r>0$, the linear span of the functions $e^{-rx}x^m$,
$m=0,1,2,\dots$, is dense in $C_0(\R_+)$.
\end{corollary}

\begin{proof}
The dual space to $C_0(\R_+)$ is the space of finite signed measures on $\R_+$.
Therefore, it suffices to prove that if $M$ is a signed measure such that
$e^{-rx}M$ is orthogonal to all polynomials, then $M=0$. To do this write $M$
as the difference of two finite positive measures $M'$ and $M''$. The
assumption on $M$ means that measures $M'(dx)e^{-rx}$ and $M''(dx)e^{-rx}$ have
the same moments. Then the argument in the proof of Lemma \ref{lemma6} shows
that these measures are equal. Therefore $M'=M''$ and $M=0$.
\end{proof}

\begin{proof}[Proof of Theorem \ref{thm2}]
It is easy to check the relations
\begin{equation}\label{eq30}
\LB^\infty_{r'}\, \LB^{r'}_{r}=\LB^\infty_{r}, \qquad  r'>r.
\end{equation}
They determine a Borel map
\begin{equation*}
\M(\R_+)\to\M_\infty=\varprojlim\M(V_r), \quad M\mapsto\{M_r\}, \quad
M_r:=M\,\LB^\infty_r.
\end{equation*}
By virtue of Remark \ref{rem2} it suffices to prove that this map is a
bijection.

By Lemma \ref{lemma6}, it is injective; even more, $M\mapsto M_r$ is injective
for any fixed $r>0$.

We proceed to the proof that the map $M\mapsto\{M_r\}$ is surjective. Fix an
element $\{M_r:r>0\}$ of the projective limit space $\M_\infty$. Let us show that
it comes from some probability measure $M\in\M(\R_+)$. The idea is that $M$
arises as a scaling limit of the measures $M_{r'}$ as $r'\to+\infty$.

Write the compatibility relation $M_{r'}\,\LB^{r'}_r=M_r$ in the form
\begin{equation}\label{eq14}
\langle M_{r'}, \, \LB^{r'}_r(\,\cdot\,,m)\rangle=M_r(m), \qquad \forall
m\in\Z_+,
\end{equation}
where $\LB^{r'}_r(\,\cdot\,,m)$ is viewed as the function $l\mapsto
\LB^{r'}_r(l,m)$ on $\Z_+$. Fix $r$ and $m$ and let parameter $r'$ go to
$+\infty$. Embed $\Z_+$ into $\R_+$ via the map
$$
\varphi_{r'}:l\mapsto x:=(1/r')l
$$
that depends on $r'$. Denote by $\wt M_{r'}$ the pushforward of $M_{r'}$ under
$\varphi_{r'}$; this is a probability measure on $\R_+$. Next, rewrite the
expression
$$
\LB^{r'}_r(l,m)=\left(1-\frac{r}{r'}\right)^{l-m}\left(\dfrac{r}{r'}\right)^m\,
\dfrac{l!}{(l-m)!m!}
$$
as a function of variable $x:=\varphi_{r'}(l)$:
\begin{equation}\label{eq33}
\LB^{r'}_r(l,m)=\frac{r^m}{m!}\cdot
\left(1-\frac{r}{r'}\right)^{r'x-m}x\left(x-\frac1{r'}\right)\dots\left(x-\frac{m-1}{r'}\right)
\end{equation}
Here $x$ ranges over the grid $\varphi_{r'}(\Z_+)=(1/r')\Z_+\subset\R_+$, but
the expression in the right-hand side of \eqref{eq33} makes sense for all
$x\in\R_+$. By Lemma \ref{lemma5}, this expression, as a function of variable
$x\in\R_+$, belongs to $C_0(\R_+)$ and converges, as parameter $r'$ goes to
$+\infty$, to the function
$$
x\mapsto e^{-rx}\frac{(rx)^m}{m!}=\LB^\infty_r(x,m)
$$
in the metric of $C_0(\R_+)$.

On the other hand, the set of sub-probability measures on $\R_+$ is compact in
the vague topology (the topology of convergence on functions from $C_0(\R_+)$).
Therefore, the family $(\wt M_{r'})$ has a nonempty set of partial vague limits
as $r'\to+\infty$. Choose any such limit $M$. Then we may pass to a limit in
\eqref{eq14}, which gives us
$$
\langle M, \, \LB^\infty_r(\,\cdot\,,m)\rangle=M_r(m), \qquad \forall m\in\Z_+,
\quad \forall r>0,
$$
which in turn implies that $M$ is actually a probability measure. This
concludes the proof of the theorem.
\end{proof}

The following example is used below in Section \ref{sc:zmeasures}.

\begin{example}\label{example1}
Fix parameter $c>0$. For any $r>0$ define a probability measure $\MB^{(c)}_r$
on $\Z_+$ by
$$
\MB^{(c)}_r(m)=(1+r)^{-c}\frac{(c)_m}{m!}\,\left(\frac r{1+r}\right)^m, \qquad
m\in\Z_+,
$$
where $(c)_m:=c(c+1)\dots(c+m-1)$. This is a negative binomial distribution. A
direct check shows that the family $\{\MB^{(c)}_r\}_{r>0}$ is compatible with
the links $\LB^{r'}_r$, so that this family is an element of the projective
limit space $\M_\infty$ associated with the system $\mathbb B$. The
corresponding limit measure on the boundary $\pd\mathbb B=\R_+$ is the gamma
distribution with parameter $c$; it has density $(\Ga(c))^{-1}x^{c-1}e^{-x}$
with respect to the Lebesgue measure.
\end{example}

\subsection{Thoma's simplex, Thoma's cone, and symmetric functions}
The {\it Thoma simplex\/} is the subspace $\Om$ of the infinite product space
$\R_+^\infty\times\R_+^\infty$ formed by all couples $(\al,\be)$, where
$\al=(\al_i)$ and $\be=(\be_i)$ are two infinite sequences such that
\begin{equation}\label{eq20}
\al_1\ge\al_2\ge\dots\ge0, \qquad \be_1\ge\be_2\ge\dots\ge0
\end{equation}
and
\begin{equation}\label{eq21}
\sum_{i=1}^\infty\al_i+\sum_{i=1}^\infty\be_i\le 1.
\end{equation}
We equip $\Om$ with the product topology inherited from
$\R_+^\infty\times\R_+^\infty$. Note that in this topology, $\Om$ is a compact
metrizable space.

The {\it Thoma cone\/} $\wt\Om$ is the subspace of the infinite product space
$\R_+^\infty\times\R_+^\infty\times\R_+$ formed by all triples
$\om=(\al,\be,\de)$, where $\al=(\al_i)$ and $\be=(\be_i)$ are two infinite
sequences and $\de$ is a nonnegative real number, such that the couple
$(\al,\be)$ satisfies \eqref{eq20} and the following modification of the
inequality \eqref{eq21}
$$
\sum_{i=1}^\infty\al_i+\sum_{i=1}^\infty\be_i\le \de.
$$
We set $|\om|=\de$.

Note that $\wt\Om$ is a locally compact space in the product topology inherited
from $\R_+^\infty\times\R_+^\infty\times\R_+$. The space $\wt\Om$ is also
metrizable and has countable base. Every subset of the form $\{\om\in\wt\Om:
|\om|\le\const\}$ is compact. Therefore, a sequence of points $\om_n$ goes to
infinity in $\wt\Om$ if and only if $|\om_n|\to\infty$.

We will identify $\Om$ with the subset of $\wt\Om$ formed by triples
$\om=(\al,\be,\de)$ with $\de=1$. The name ``Thoma cone'' given to $\wt\Om$ is
justified by the fact that $\wt\Om$ may be viewed as the cone with the base
$\Om$: the ray of the cone passing through a base point $(\al,\be)\in\Om$
consists of the triples $\om=(r\al,r\be,r)$, $r\ge0$.

More generally, for $\om=(\al,\be,\de)\in\wt\Om$ and $r>0$ we set
$r\om=(r\al,r\be,r\de)$.

Let $\Sym$ denote the graded algebra of symmetric functions over the base field
$\R$ (see, e.g., \cite{Ma95}, \cite{Sa01}). As an abstract algebra, $\Sym$ is
isomorphic to the polynomial algebra $\R[p_1,p_2,\dots]$, where the generators
$p_k$ are the power sums in formal variables $x_1,x_2,\dots$,
$$
p_k=\sum_{i=1}^\infty x_i^k, \qquad \deg p_k=k.
$$
Here we employ the (conventional) realization of $\Sym$ as the subalgebra in
$\R[[x_1,x_2,\dots]]$ formed by symmetric power series in countably many
variables, of bounded total degree, see \cite{Sa01}.

However, this realization is not used in what follows. Instead, we embed $\Sym$
into the algebra of continuous functions on the Thoma cone by setting
$$
p_k(\om)=\begin{cases}\sum_{i=1}^\infty \al_i^k+ (-1)^{k-1}\sum_{i=1}^\infty
\be_i^k, & k=2,3,\dots\\
|\om|, & k=1,
\end{cases}
$$
where $\om$ ranges over $\wt\Om$.

In more detail, every element $F\in\Sym$ is uniquely written as a polynomial in
$p_1,p_2,\dots$; then we define $F(\om)$ as the same polynomial in numeric
variables $p_1(\om), p_2(\om),\dots$\,. Note that the above expressions with
$k\ge2$ are the super power sums in variables $(\al_i)$ and $(-\be_i)$, see
\cite[\S I.3, Ex. 23]{Ma95}.

Another system of generators in $\Sym$ is provided by the {\it complete
homogeneous symmetric functions} $h_1,h_2,\dots$ whose relation with $p_k$'s
can be conveniently written in the form
$$
H(t)=\exp(P(t)),
$$
where $H(t)=1+\sum_{k\ge 1} h_kt^k$ and $P(t)=\sum_{k\ge 1}p_kt^k/k$ are
suitable generating functions.

Hence, under the embedding of $\Sym$ into $C(\wt\Omega)$ described above, we have
\begin{equation}\label{eq:add2}
1+h_1(\omega)t+h_2(\omega)t^2+\dots=e^{\gamma t} \prod_{i=1}^\infty \frac{1+\beta_i t}{1-\alpha_i t}\,,
\end{equation}
where $\omega=(\alpha,\beta,\delta)$, and $\gamma:=\delta-\sum_{i\ge 1}(\alpha_i+\beta_i)\ge 0$.

A distinguished {\it linear} basis of $\Sym$ is formed by the {\it Schur
functions\/}. We denote them by $S_\mu$, where the index $\mu$ ranges over
$\Y$. The Schur functions are homogeneous elements, $\deg S_\mu=|\mu|$, and
they can be expressed through the complete homogeneous symmetric functions by
the {\it Jacobi-Trudi formula}
$$
S_\mu=\det\bigl[ h_{\mu_i-i+j}\bigr]_{i,j=1}^\ell,
$$
where $\ell=\ell(\mu)$ is the number of nonzero parts of $\mu$, and we assume
that $h_0=1$, $h_{-1}=h_{-2}=\dots=0$. Thus, the functions $S_\mu\in C(\wt\Om)$
are given by
$$
S_\mu(\om)=\det\bigl[ h_{\mu_i-i+j}(\om)\bigr]_{i,j=1}^\ell,
$$
where $h_k(\omega)$ are determined by \eqref{eq:add2}.

\subsection{The Young graph $\Y$}\label{Young}
Consider the group chain \eqref{eq32}, where the $n$th group is the symmetric
group $S(n)$ formed by permutations of the set $\{1,\dots,n\}$. The embedding
$S(n)\subset S(n+1)$ is defined by identifying $S(n)$ with the subgroup of
$S(n+1)$ fixing the point $n+1$. The branching graph associated with this group
chain is called the {\it Young graph\/} and denoted by $\Y$. The vertices of
$\Y$ are the Young diagrams including the empty diagram $\varnothing$ at level
0. The level of a Young diagram $\la$ equals the number of its boxes, and two
diagrams are joined by a (simple) edge if they differ by a single box. This
agrees with general Definition \ref{ex2} by virtue of the {\it Young branching
rule\/} for irreducible representations of symmetric groups (see, e.g.,
\cite[Theorem 2.8.3]{Sa01}).

Young diagrams are usually identified with {\it partitions\/} and written in
the partition notation, $\la=(\la_1,\la_2,\dots)$. Here, by definition, $\la_i$
equals the number of boxes in the $i$th row of $\la$. We set $|\la|=\sum\la_i$;
this is the same as the number of boxes in the diagram $\la$.

The dimension function in the Young graph has a nice combinatorial meaning:
$\dim\la$ coincides with the number of standard tableaux of shape $\la$. For
this quantity there are several nice explicit formulas, e.g., the hook formula
(see \cite[Theorem 3.10.2]{Sa01}).

Consider the projective chain defined by the Young graph:
$$
\Y_0\dashleftarrow\Y_1\dashleftarrow\Y_2\dashleftarrow\cdots
$$
with the links $\LY^{m+1}_m:\Y_{m+1}\dasharrow\Y_m$ defined by (below
$\mu\in\Y_m$ and $\nu\in\Y_{m+1}$)
$$
\LY^{m+1}_m(\nu,\mu)=\begin{cases}\dim\mu/\dim\nu, &\mu\subset\nu,\\
0, &\textrm{otherwise}\end{cases}
$$
(the notation $\mu\subset\nu$ means that $\mu$ is a subdiagram of $\nu$; since
$|\nu|=|\mu|+1$, this is equivalent to saying that $\mu$ is obtained from $\nu$
be removing a box).

More generally, for any $n>m$ the link $\LY^n_m:\Y_n\dasharrow\Y_m$ is defined
as the composition
$$
\LY^n_m=\LY^n_{n-1}\dots\LY^{m+1}_m
$$
and has the form
\begin{equation}\label{eq34}
\LY^n_m(\nu,\mu)=\frac{\dim\mu\cdot\dim(\mu,\nu)}{\dim\nu}, \qquad \nu\in\Y_n,
\quad \mu\in\Y_m,
\end{equation}
where $\dim(\mu,\nu)$ is defined as the number of standard tableaux of skew
shape $\nu/\mu$ if $\mu\subset\nu$, and 0 otherwise.

For a Young diagram $\la$, its {\it modified Frobenius coordinates\/}
$(a_1,\dots,a_d; b_1,\dots,b_d)$ are defined as follows: $d$ is the number of
diagonal boxes in $\la$; $a_i$ is equal to $\frac12$ plus the number of boxes
in the $i$th row, on the right of the $i$th diagonal box; likewise, $b_i$ is
equal to $\frac12$ plus the number of boxes in the $i$th column, below the
$i$th diagonal box. Note that
$$
\sum_{i=1}^d (a_i+b_i)=|\la|.
$$

We embed the set $\Y$ into $\wt\Om$ through the map
$$
\la\mapsto \om_\la:=((a_1,\dots,a_d,0,0,\dots), \, (b_1,\dots,b_d,0,0,\dots),
\, |\la|).
$$
Obviously, $|\om_\la|=|\la|$.

Recall that for any $\mu\in\Y$ we denote by $S_\mu$ the corresponding Schur
symmetric function.

\begin{lemma}\label{lemma1}
In the algebra\/ $\Sym$, there exist elements $\FS_\mu$ indexed by diagrams
$\mu\in\Y$, such that
$$
FS_\mu=S_\mu + \textrm{lower degree terms}
$$
and
\begin{equation}\label{eq13}
l^{\down m}\frac{\dim(\mu,\la)}{\dim\nu}=\FS_\mu(\om_\la), \qquad \la\in\Y,
\quad l=|\la|,
\end{equation}
where
$$
l^{\down m}=l(l-1)\dots(l-m+1).
$$
\end{lemma}

\begin{proof} See \cite[Section 2]{ORV03}. The result is actually a reformulation of
\cite[Theorem 8.1]{OO97}. The elements $\FS_\mu$ are called the {\it
Frobenius-Schur functions\/}.
\end{proof}

\begin{corollary}\label{cor2}
Fix\/ $m$ and $\mu\in\Y_m$. For large $l$ and $\la\in\Y_l$
$$
\frac{\dim(\mu,\la)}{\dim\nu}=S_\mu(l^{-1}\om_\la)+O(l^{-1}),
$$
where the bound $O(l^{-1})$ for the rest term depends on $m$ and $\mu$ but is
uniform on $\la$.
\end{corollary}

\begin{proof}
Observe that for any homogeneous element $F\in\Sym$, one has
\begin{equation}\label{eq15}
F(\om)=O\left(|\om|^{\deg F}\right),
\end{equation}
where the bound depends only on $F$. Indeed, it suffices to check this for the
generators $p_k$ and then the assertion is immediate from the very definition
of $p_k(\om)$.

By Lemma \ref{lemma1}, the expansion of $\FS_\mu$ on homogeneous components has
the form
$$
\FS_\mu=S_\mu+\sum_{k=0}^{m-1}F_k
$$
where $F_0,\dots,F_{m-1}$ are some homogeneous elements with $\deg F_k=k$;
their explicit form is inessential. Hence,
$$
\FS_\mu(\om_\la)=S_\mu(\om_\la)+\sum_{k=0}^{m-1}F_k(\om_\la) =l^m\left(
S_\mu(l^{-1}\om_\la)+\sum_{k=0}^{m-1}\frac1{l^{m-k}}F_k(l^{-1}\om_\la)\right).
$$

Therefore,
$$
\frac{\dim(\mu,\la)}{\dim\nu}=\dfrac1{l^{\down
m}}\FS_\mu(\om_\la)=\dfrac{l^m}{l^{\down m}}\left(
S_\mu(l^{-1}\om_\la)+\sum_{k=0}^{m-1}\frac1{l^{m-k}}F_k(l^{-1}\om_\la)\right).
$$
Taking into account equality $l=|\om_\la|$ and applying \eqref{eq15} we see
that the asymptotics of this expression is indeed
$S_\mu(l^{-1}\om_\la)+O(l^{-1})$.
\end{proof}

\begin{lemma}\label{lemma2}
For $m=1,2,\dots$ there exist links $\LY^\infty_m:\Om\dasharrow \Y_m$ defined
by
$$
\LY^\infty_m(\om,\mu)=\dim\mu\cdot S_\mu(\om), \qquad \om\in\Om, \quad
\mu\in\Y_m.
$$
They satisfy the compatibility relation
\begin{equation}\label{eq31}
\LY^\infty_{m+1}\LY^{m+1}_m=\LY^\infty_m, \qquad m=1,2,\dots\,.
\end{equation}
\end{lemma}

\begin{proof}
The key observation is that any point $\om\in\Om$ can be approximated by an
appropriate sequence of points of the form $l^{-1}\om_\la$, where $l\to\infty$
and $\la\in\Y_l$ varies together with $l$. Fix $m$ and $\mu\in\Y_m$. Since the
function $S_\mu(\om)$ is continuous on $\Om$, the preceding lemma implies that
$\LY^\infty_m(\om,\mu)\ge0$ for any $\om\in\Om$.

The same approximation argument shows that
$$
\sum_{\mu\in\Y_m}\LY^\infty_m(\om,\mu)=1,
$$
because the sum is finite and the similar relation holds for $\LY^l_m$.

Likewise, the limit transition as $l\to\infty$ in
$$
\sum_{\nu\in\Y_{m+1}}\LY^l_{m+1}(\la,\nu)\LY^{m+1}_m(\nu,\mu)=\LY^l_m(\la,\mu),
$$
proves the required compatibility relation.
\end{proof}

\begin{theorem}\label{thm3}
The links $\LY^\infty_m:\Om\dasharrow \Y_m$ introduced above make it possible
to identify the boundary $\pd\Y$  of the Young graph\/ $\Y$ with the Thoma
simplex\/ $\Om$.
\end{theorem}

\begin{proof}
We use the same argument as in the proof of Theorem \ref{thm2}. To some extent,
the situation is even simpler because $\Om$ is a compact space.

By virtue of Lemma \ref{lemma2} the links $\LY^\infty_m:\Om\dasharrow \Y_m$
define a map
$$
\M(\Om)\to \M_\infty:=\varprojlim\M(\Y_m),
$$
and we have to check that is bijective.

The functions $F(\om)$ on $\Om$ coming from elements $F\in\Sym$ form a real
algebra that contains 1 and separates points. By Stone-Weierstrass' theorem, this
algebra is dense in $C(\Om)$. Hence, every measure $M$ on $\Om$ is
uniquely determined by its pairings $\langle M,F\rangle$. Since the Schur
functions $S_\mu$ form a basis in $\Sym$, $M$ is uniquely determined by its
pairings with the functions $\LY^\infty_m(\,\cdot\,,\mu)$. This proves
injectivity.

To prove surjectivity, fix an element $(M_m)\in\M_\infty$. For each $l$
consider the embedding
$$
\varphi_l:\Y_l\to\Om, \qquad \varphi_l(\la):=l^{-1}\om_\la, \quad \la\in\Y_l.
$$
It takes $M_l$ to a probability measure $\wt M_l$ on $\Om$.

Next, by virtue of Lemma \ref{lemma1}, the compatibility relation
$M_l\LY^l_m=M_m$ can be rewritten as
$$
\dim\mu \, \langle \wt M_l, S_\mu\rangle+O(l^{-1})=M_m(\mu).
$$
Let $M$ stand for any partial weak limit of the sequence $(\wt M_l)$ as
$l\to\infty$. Then the above relation implies
$$
\dim \mu\,\langle M, S_\mu\rangle=M(\mu)
$$
which is equivalent to $M\,\LY^\infty_m=M_m$. This concludes the proof.
\end{proof}

This result is closely related to Thoma's theorem on the characters of the
infinite symmetric group $S(\infty)$. The above proof follows the approach of
the paper Kerov, Okounkov, and Olshanski \cite{KOO98}, which in turn develops
the ideas of Vershik and Kerov \cite{VK81}, \cite{VK90}; see also Kerov'
monograph \cite{Ke03}.

\subsection{The Young bouquet $\YB$}
The set $\Y$ of Young diagrams is a poset with respect to the partial order
defined by inclusion of diagrams. That is, a diagram $\mu$ is smaller than a
diagram $\nu$ if $\mu$ is contained in $\nu$. Equivalently, in the partition
notation, $\mu_i\le\nu_i$ for all $i$, where at least one inequality is strict.
As a poset, $\Y$ is a lattice, and for this reason it is often called the {\it
Young lattice\/}. There is an obvious relation between the order on $\Y$ and
the graph structure.

We are going to define a (partially) continuous analog of the Young lattice
$\Y$.

\begin{definition}
The {\it Young bouquet\/} is the poset $(\YB,<)$ defined as follows.

1) The set $\YB$ is the wedge sum of countably many rays indexed by all Young
diagrams $\mu\in\Y$ (whence the term ``bouquet'', which is a synonym for wedge
sum). More precisely, $\YB$ is obtained from the direct product space
$\Y\times\R_+$ (where $\R_+=[0,+\infty)$) by gluing together all the points
$(\mu,0)$ into a single point, denoted as $(\varnothing,0)$.

2) The partial order in $\YB$ comes from the conventional partial order in the
Young lattice $\Y$ and the conventional order in $\R_+$. That is, an element
$(\mu, r)\in\YB$ is declared to be smaller than another element $(\nu, r')$ if
$r<r'$ and $\mu\subseteq\nu$; then we write $(\mu,r)<(\nu,r')$ or
$(\nu,r')>(\mu,r)$.
\end{definition}

For an element $(\mu,r)\in\YB$ we write $|(\mu,r)|:=r$ and call this number the
{\it level\/} of $(\mu,r)$. Let $\YB_r$ denote the subset of elements of level
$r$. The stratification $\YB=\sqcup_{r\ge0}\YB_r$ is viewed as a continuous
analog of grading. Unless otherwise stated, below we assume $r>0$ and
identify each level set $\YB_r$ with $\Y$.

\begin{definition}
With every couple $r'>r>0$ we associate a matrix $\LYB^{r'}_{r}$ of format
$\Y\times\Y$, with the entries
\begin{align}
\LYB^{r'}_{r}(\nu, \mu)&=\LB^{r'}_r(n,m)\LY^n_m(\nu,\mu)  \label{eq28}\\
&=\left(1-\frac{r}{r'}\right)^{n-m}\left(\dfrac{r}{r'}\right)^m\,
\dfrac{n!}{(n-m)!\,m!}\,\dfrac{\dim\mu\, \dim(\mu,\nu)}{\dim\nu},  \label{eq29}
\end{align}
where $n:=|\nu|$ and $m:=|\mu|$ and the matrices right-hand side of
\eqref{eq28} are defined in \eqref{eq11} and \eqref{eq34}.
\end{definition}

Due to the factor $(n-m)!$ in the denominator and the factor $\dim(\mu,\nu)$ in
the numerator $\LYB^{r'}_{r}(\nu, \mu)$ vanishes unless $(\mu,r)<(\nu,r')$.

{}From \eqref{eq28} one sees that $\LYB^{r'}_{r}$ is a stochastic matrix,
because it is composed from two auxiliary stochastic matrices. In other words,
given $\nu$, the random diagram $\mu$ can be drown in two steps: First, we
choose its size $m$ according to the binomial distribution
$\LB^{r'}_r(n,\,\cdot\,)$ and then $\mu$ is specified inside $\Y_m$ according
to the probabilities from the second stochastic matrix. Thus, $\LYB^{r'}_{r}$
is a link $\Y\dasharrow\Y$.

The new links satisfy the compatibility relation
$$
\LYB^{r''}_{r'}\,\LYB^{r'}_r=\LYB^{r''}_r, \qquad r''>r'>r,
$$
because the auxiliary links satisfy analogous compatibility relations. Thus, we
get a projective system formed by the levels $\YB_r=\Y$, $r>0$, of the Young
bouquet with the links $\LYB^{r'}_r$. By definition, the {\it boundary of the Young
bouquet\/} is the boundary of this projective system. We aim to show that this
boundary is the Thoma cone $\wt\Om$.

Let $(0,0,0)$ denote the origin of the Thoma cone; this is the only point
$\om\in\wt\Om$ with $|\om|=0$. To every $\om\in\wt\Om\setminus\{(0,0,0)\}$ we
assign the point $\wh\om=|\om|^{-1}\om$ in the Thoma simplex $\Om$. The map
$\om\mapsto (|\om|,\,\wh\om)$ is a bijection between $\Om\setminus\{(0,0,0)\}$
and the ``cylinder'' $\R_{>0}\times \Om$.

\begin{definition}
Let $r>0$. For  $\om=(x,\wh\om)\in\wt\Om\setminus\{(0,0,0)\}$ and $\mu\in\Y_m$
we set
\begin{align*}
\LYB^\infty_r(\om,\mu)&=\LB^\infty_r(x, m)\,\LY^\infty_m(\wh\om,\mu)\\
&=e^{-rx}\,\frac{(rx)^m}{m!}\,\dim\mu\cdot S_\mu(\wh\om)\\
&=e^{-r|\om|}\,\frac{r^m}{m!}\,\dim\mu\cdot S_\mu(\om).
\end{align*}
We extend this definition to the origin $\om=(0,0,0)$ by continuity, which
gives
\begin{equation*}
\LYB^\infty_r((0,0,0),\mu)=\begin{cases} 1, & \mu=\varnothing,\\
0, &\mu\ne\varnothing. \end{cases}
\end{equation*}
\end{definition}

\begin{lemma}\label{lemma4}
For every $r>0$, $\LYB^\infty_r$ is a link\/ $\wt\Om\dasharrow\Y$.

{\rm(ii)} The links $\LYB^\infty_r$ satisfy the compatibility relation
$$
\LYB^\infty_{r'}\,\LYB^{r'}_r=\LYB^\infty_r, \qquad r'>r>0.
$$
\end{lemma}

\begin{proof}
(i) We have to check that $\LYB^\infty_r(\om,\,\cdot\,)$ is a probability
measure on $\Y$ for every $\om\in\wt\Om$. Consider separately the cases
$\om\ne(0,0,0)$ and $\om=(0,0,0)$. In the first case, the claim follows from
the factorization of $\LYB^\infty_r(\om,\mu)$; this quantity is represented as
the probability to select $\mu$ through a 2-step procedure directed by two
probability distributions.
In the second case the claim is obvious, for $\LYB^\infty_r((0,0,0),
\,\cdot\,)$ is the delta measure at $\mu=\varnothing$.

(ii) We have to check that
$$
(\LYB^\infty_{r'}\,\LYB^{r'}_r)(\om,\,\cdot\,)=\LYB^\infty_r(\om,\,\cdot\,),
\qquad r'>r>0,
$$
for any $\om\in\wt\Om$. Consider again the same two case: $|\om|=0$ and
$|\om|>0$. In the first case, both sides are delta measures at
$\varnothing\in\Y$. In the second case we use the factorization property of the
links $\LYB^\infty_r$ and $\LYB^{r'}_r$ and the compatibility relations for the
auxiliary links, see \eqref{eq30} and \eqref{eq31}.
\end{proof}

\begin{lemma}[cf. Lemma \ref{lemma6}]\label{lemma3}
Fix $r>0$ and assume that $M'$ and $M''$ are two finite Borel measures on
$\wt\Om$ such that
\begin{equation}\label{eq:add1}
\int_{\wt\Om}M'(d\om)e^{-r|\om|}F(\om)=\int_{\wt\Om}M''(d\om)e^{-r|\om|}F(\om)
\end{equation}
for all $F\in\Sym$. Then $M'=M''$.
\end{lemma}

\begin{proof}
{\it Step\/} 1. Let $\bar M'$ and $\bar M''$ stand for the pushforwards of $M'$
and $M''$ under the projection $\wt\Om\to\R_+$ defined as $\om\mapsto|\om|$. We
claim that $\bar M'=\bar M''$.

Indeed, recall that $p_1(\om)=|\om|$. Taking $F=p_1^k$ we get
$$
\int_{\R_+}\bar M'(dx)e^{-rx}x^k=\int_{\R_+}\bar M''(dx)e^{-rx}x^k, \qquad
\forall k\in\Z_+.
$$
Now the argument of Lemma \ref{lemma6} shows that $\bar M'=\bar M''$.

{\it Step\/} 2. Without loss of generality we may assume that $M'$ and $M''$
have no atom at the origin of the Thoma cone. Indeed, if $M'$ has an atom at
the origin, then $\bar M'$ has an atom of the same mass at the point
$0\in\R_+$. Since $\bar M'=\bar M''$, the measure $M''$ has the same atom as
$M'$, so that we may simply remove it.

{\it Step\/} 3. The previous step allows us to transfer the measures $M'$ and
$M''$ from the cone $\wt\Om$ to the cylinder $\R_{>0}\times\Om$ with
coordinates $(x,\wh\om)$, where $x=|\om|\in\R_{>0}$ and
$\wh\om=|\om|^{-1}\om\in\Om$. Step 1 tells us that the projections of the both
measures on coordinate $x$ are one and the same measure $\bar M:=\bar M'=\bar
M''$ on $\R_{>0}$. Let us disintegrate $\bar M'$ and $\bar M''$ with respect to
$\bar M$ (see e.g. Theorem 8.1 in \cite{Pa67} on the existence of the
conditional distributions). Then we get two families $\{Q'_x\}$ and $\{Q''_x\}$
of probability measures on $\Om$, indexed by points $x\in\R_{>0}$. These
families are defined uniquely, modulo $\bar M$-null sets.

We claim that
\begin{equation}\label{eq12}
\int_\Om Q'_x(d\wh\om) G(\wh\om)=\int_\Om Q''_x(d\wh\om) G(\wh\om)
\end{equation}
for all $G\in\Sym$ and all $x$ outside an appropriate $\bar M$-null set that
does not depend on $G$.

Indeed, since $\Sym$ possesses a countable homogeneous basis (for instance, the
basis of Schur functions) it suffices to prove that \eqref{eq12} holds for any
given homogeneous function $G$ and for all $x$ outside a $\bar M$-null set
possibly dependent on $G$.

Then substitute $F=p_1^k G$ into the initial equality \eqref{eq:add1} and denote by $m$ the
degree of $G$. We get the equality
$$
\int_{\R_+}\bar M(dx)e^{-rx}x^{m+k}\left\{\int_\Om
Q'_x(d\wh\om)G(\wh\om)\right\}=\int_{\R_+}\bar
M(dx)e^{-rx}x^{m+k}\left\{\int_\Om Q''_x(d\wh\om)G(\wh\om)\right\},
$$
which holds for any $k\in\Z_+$. This means equality of moments for two
measures, each of which is is the product of $\bar M(dx)e^{-rx}x^m$ and a
bounded function. The same argument as in step 1 shows that these two measures
are the same, which proves \eqref{eq12}.

{\it Step\/} 4. The functions from $\Sym$ are dense in $C(\Om)$ because they
separate points and the space $\Om$ is compact. Together with \eqref{eq12} this
implies that $Q'_x=Q''_x$ almost everywhere with respect to $\bar M$. We
conclude that $M'=M''$.
\end{proof}

Recall that $\wt\Om$ is a locally compact space. Let $C_0(\wt\Om)$ stand for
the real Banach space of continuous functions on $\wt\Om$ vanishing at
infinity, with the supremum norm.

\begin{corollary}[cf. Corollary \ref{cor3}]\label{cor1}
For any fixed $r>0$, the set of functions of the form $e^{-rx}F$ with $F$
ranging over\/ $\Sym$ is dense in $C_0(\wt\Om)$.
\end{corollary}

\begin{proof}
We argue as in Corollary \ref{cor3}, with appeal to \ref{lemma3} instead of
Lemma \ref{lemma6}.
\end{proof}

\begin{theorem}\label{thm4}
The Thoma cone $\wt\Om$ together with the collection of links
$\LYB^\infty_r:\wt\Om\dasharrow\Y$, $r>0$, is the boundary of the Young
bouquet.
\end{theorem}

\begin{proof}
We follow the scheme of the proof of Theorem \ref{thm2}. By virtue of Lemma
\ref{lemma4}, the links $\LYB^\infty_r$ define a Borel map $M\mapsto
(M_r)_{r>0}$ from $\M(\wt\Om)$ to the projective limit space constructed from
the system $\{V_r=\Y, \LYB^{r'}_r\}$. According to Remark \ref{rem2}, it
suffices to prove that this map is a bijection. We divide this claim into two
parts, injectivity and surjectivity.

The injectivity claim follows from Lemma \ref{lemma3} or Corollary \ref{cor1},
which say that even the map $M\mapsto M_r$ with any fixed $r>0$ is injective.

We proceed to the proof of the surjectivity claim. Write the compatibility
relation $M_{r'}\LYB^{r'}_r=M_r$ in the form
$$
\langle M_{r'}, \, \LYB^{r'}_r(\,\cdot\,,\mu)\rangle=M_r(\mu), \qquad \forall
\mu\in\Y,
$$
where $\LYB^{r'}_r(\,\cdot\,,\mu)$ is viewed as the function $\la\mapsto
\LYB^{r'}_r(\la,\mu)$ on $\Y$. Fix $r$ and $\mu$ and let parameter $r'$ go to
$+\infty$. Embed $\Y$ into $\wt\Om$ via the map
$$
\varphi_{r'}:\la\mapsto (1/r')\om_\la
$$
that depends on $r'$. Denote by $\wt M_{r'}$ the pushforward of $M_{r'}$ under
$\varphi_{r'}$; this is a probability measure on $\wt\Om$. Next, regard
$\LYB^{r'}_r(\la,\mu)$ as a function of variable $\om:=\varphi_{r'}(\la)$. Of
course, this function is initially defined only on the discrete subset
$\varphi_{r'}(\Y)\subset\wt\Om$, but we will see that it admits a natural
extension to a continuous function on the whole $\wt\Om$ depending also on
parameter $r'$. The key fact proved below is that the latter function lies in
the Banach space $C_0(\wt\Om)$ and converges, as $r'\to\infty$, to the function
$\om\mapsto\LYB^\infty(\om,\mu)$ in the metric of that space. Once this is
established, the desired result follows. Indeed, take as $M$ an arbitrary
partial limit of $\{\wt M_{r'}, r'\to+\infty\}$ with respect to the vague
topology; this is a sub-probability measure on $\wt\Om$. Then we may pass to
the limit in the above equation, which gives us
$$
\langle M, \, \LYB^\infty_r(\,\cdot\,,\mu)\rangle=M_r(\mu), \qquad \forall
\mu\in\Y, \quad \forall r>0,
$$
which in turn implies that $M$ is a probability measure. This concludes the
proof modulo the claim concerning the  function $\LYB^{r'}_r(\,\cdot\,,\mu)$
and its convergence to $\LYB^\infty_r(\,\cdot\,,\mu)$ in the metric of
$C_0(\wt\Om)$.

Now let us prove that claim. Write again the explicit expression for
$\LYB^{r'}_r(\la,\mu)$:
$$
\LYB^{r'}_r(\la,\mu)=\left(1-\frac{r}{r'}\right)^{l-m}\left(\dfrac{r}{r'}\right)^m\,
\dfrac{l!}{(l-m)!\,m!}\,\dfrac{\dim\mu\, \dim(\mu,\la)}{\dim\la},
$$
where, as usual, $l=|\la|$ and $m=|\mu|$. By virtue of \eqref{eq13},
$$
\dfrac{l!}{(l-m)!}\,\dfrac{\dim\mu\,
\dim(\mu,\la)}{\dim\la}=\FS_\mu(\om_\la)=S_\mu(\om_\la)+\sum_{k=0}^{m-1}F_k(\om_\la),
$$
where $F_0,\dots,F_{m-1}$ are the same homogeneous elements of $\Sym$ with
$\deg F_k=k$ as in the proof of Corollary \ref{cor2}.

Setting $\om:=(1/r')\om_\la$ (and keeping in mind that $\om$ depends both on
$\la$ and $r'$) we may rewrite the above equality as
$$
\frac1{(r')^m}\dfrac{l!}{(l-m)!}\,\dfrac{\dim\mu\,
\dim(\mu,\la)}{\dim\la}=S_\mu(\om)+\sum_{k=0}^{m-1}\frac1{(r')^{m-k}}F_k(\om).
$$
Now, returning to $\LYB^{r'}_r(\la,\mu)$, we may write it as
$$
\LYB^{r'}_r(\la,\mu)=\frac{r^m\dim\mu}{m!}\cdot\left(1-\frac{r}{r'}\right)^{l-m}
\left(S_\mu(\om)+\sum_{k=0}^{m-1}\frac1{(r')^{m-k}}F_k(\om)\right).
$$
Since $l=|\om_\la|=r'|\om|$, we finally get a nice formula
$$
\LYB^{r'}_r(\la,\mu)=\frac{r^m\dim\mu}{m!}\cdot\left(1-\frac{r}{r'}\right)^{r'|\om|-m}
\left(S_\mu(\om)+\sum_{k=0}^{m-1}\frac1{(r')^{m-k}}F_k(\om)\right).
$$

In this formula $\om$ is assumed to be related to $\la$ via relation
$\om=(1/r')\om_\la$ but the right-hand side is well defined as a function on
the whole space $\wt\Om$. We have to prove that this function is continuous,
vanishes at infinity, and in the limit as $r'\to\infty$ converges to
$$
\LYB^\infty_r(\om,\mu)=\frac{r^m\dim\mu}{m!}\cdot e^{-r|\om|} S_\mu(\om)
$$
in the metric of the Banach space $C_0(\wt\Om)$. But this follows from Lemma
\ref{lemma5} by virtue of the bound \eqref{eq15}.
\end{proof}

\subsection{Z-Measures on $\YB$}\label{sc:zmeasures}
Introduce some notation. For $z\in\C$ and $\mu\in\Y$, set
\begin{equation*}
(z)_{\mu}=\prod_{(i,j)\in\mu}(z+j-i),
\end{equation*}
where the product is taken over the boxes $(i,j)$ of diagram $\mu$ (here $i$
are $j$ stand for the row and column numbers of the box). This is a
generalization of the Pochhammer symbol: In the particular case when $\mu=(m)$
is a one-row diagram, we get $(z)_\mu=(z)_m=z(z+1)\dots(z+m-1)$.

\begin{definition}\label{def1}
Let us say that a couple $(z,z')\in\C^2$ of complex parameters is {\it
admissible\/} if $z\ne0$, $z'\ne0$, and  $(z)_\mu(z')_\mu\ge0$ for all
$\mu\in\Y$.
\end{definition}

Obviously, the set of admissible values is invariant under symmetries
$(z,z')\to(z',z)$ and $(z,z)\to(-z,-z')$; the latter holds because
$(-z)_\mu=(-1)^{|\mu|}(z)_{\mu'}$.

It is not difficult to get an explicit description of the admissible range of
the parameters $(z,z')$, see \cite[Proposition 1.2]{BO06}. One can represent it
as the union of the following three subsets or {\it series\/}:

\begin{itemize}

\item The {\it principal series\/} is $\{(z,z')\colon z'=\bar
z\in\C\setminus\Z\}$.

\item The {\it complementary series\/} is $\cup_{k\in\Z}\{(z,z')\colon
k<z,z'<k+1\}$.

\item The {\it degenerate series\/} comprises the set
$$
\{(z,z')=(k,k+b-1)\colon k=1,2,\dots;\, b>0\}
$$
together with its images under the symmetry group $\Z_2\times\Z_2$.

\end{itemize}

The reason why  the values $z=0$ and $z'=0$ are excluded is that then
$(z)_\mu(z')_\mu$ vanishes for all $\mu\ne\varnothing$, which is a trivial
case. Note that $zz'>0$ for any admissible couple $\z$.

\begin{definition}
The {\it z-measure\/} with admissible parameters $\z$ and additional parameter
$r>0$ is the measure on $\Y$ given by
$$
\MYB^\z_r(\mu)=(1+r)^{-zz'}(z)_\mu(z')_\mu\left(\frac
r{1+r}\right)^{|\mu|}\left(\frac{\dim\mu}{|\mu|!}\right)^2, \qquad \mu\in\Y.
$$
\end{definition}

\begin{proposition}
The z-measures are probability measures, and they are compatible with the links
$\LYB^{r'}_r:\Y\dasharrow\Y${\rm:}
$$
\MYB^\z_{r'}\,\LYB^{r'}_r= \MYB^\z_r, \qquad r'>r>0.
$$
\end{proposition}

\begin{proof}
Set $c=zz'$ and $m=|\mu|$. The measure $\MYB^\z_r$ can be written in the form
$$
\MYB^\z_r(\mu)=\MB^{(c)}_r(m)\,\MY^\z_m(\mu),
$$
where the first factor in the right-hand side has been defined in Example
\ref{example1} and the second factor is defined by
$$
\MY^\z_m(\mu)=\frac{(z)_\mu(z')_\mu}{(c)_m}\,\frac{(\dim\mu)^2}{m!}.
$$
It is known that for each $m\in\Z_+$, $\MY^\z_m$ is a probability measure on
$\Y_m$ and the family $\{\MY^\z_m\}_{m\in\Z_+}$ is compatible with the links
$\LY^n_m:\Y_n\dasharrow Y_m$, see \cite{Ols03a}, \cite{BO00b}. Together with
Example \ref{example1} this implies the proposition.
\end{proof}

The z-measures play a key role in harmonic analysis on the infinite symmetric
group, see the survey \cite{Ols03b} and references therein.

\section{Connection with the Gelfand--Tsetlin graph}\label{sect4}

\subsection{The Gelfand--Tsetlin graph $\GT$}\label{GT}
For $N=1,2,\dots$ define a {\it signature\/} of length $N$ as an $N$-tuple of
nonincreasing integers $\mu=(\mu_1\ge\dots\ge \mu_N)\in\Z^N$, and denote by
$\GT_N$ the set of all such signatures. Elements of $\GT_N$ parameterize
irreducible representations of the compact unitary group $U(N)$ (``signature''
is another name for ``highest weight'' in the special case of the group $U(N)$,
see, e.g., \cite{Wey39}, \cite{Zhe70}.) We will also use for elements
$\mu\in\GT_N$ a more detailed notation $[\mu,N]$.

Write  $[\mu,N]\prec[\nu,N+1]$ if $\nu_j\ge \mu_j\ge\nu_{j+1}$ for all
meaningful values of indices. These inequalities are well-known to be
equivalent to the condition that the restriction of the $\nu$-representation of
$U(N+1)$ to $U(N)$ contains a $\mu$-component (then the multiplicity of this
component equals 1).

\begin{definition}
Set $\GT=\bigsqcup_{N\ge 1} \GT_N$, and equip $\GT$ with edges that join any
two signatures $\mu$ and $\nu$ such that $\mu\prec\nu$ or $\nu\prec\mu$. This
turns $\GT$ into a graph that is called the {\it Gelfand-Tsetlin graph\/}. It
will be denoted by the same symbol $\GT$.
\end{definition}

By the very definition, $\GT$ is a branching graph with countable levels. It
arises from the chain $U(1)\subset U(2)\subset\cdots$ of compact unitary groups
just as the Young graph arises from the chain of symmetric groups $S(1)\subset
S(2)\subset\cdots$. As in the Young graph $\Y$, all edges in $\GT$ are simple;
this is because the restriction of an irreducible representation of $U(N+1)$ to
the subgroup $U(N)$ is always multiplicity free.

The dimension function in $\GT$ will be denoted by the symbol $\Dim$. We have
$$
\Dim[\mu,N]=\prod_{1\le i<j\le N}\frac{\mu_i-\mu_j-i+j}{j-i}.
$$
This is classical Weyl's formula for the dimension of irreducible
representations of the unitary groups.

More generally, for $N'>N$ we write $\Dim([\mu,N],[\nu,N'])$ for the relative
dimension. According to the general definition \eqref{eq22}, the links between
various levels of $\GT$ have the form
\begin{equation}\label{eq23}
\LGT^{N'}_N([\nu,N'],[\mu,N])=\frac{\Dim[\mu,N]\,\Dim([\mu,N],[\nu,N'])}{\Dim[\nu,N']}.
\end{equation}

\subsection{The boundary of the Gelfand-Tsetlin graph} Consider the space
$\wt\Om\times\wt\Om$, the direct product of two copies of the Thoma cone. Its elements
are pairs $(\omega^+,\om^-)$, where $\om^\pm=(\al^\pm,\be^\pm,\delta^\pm)\in\wt\Om$.
It is convenient to introduce auxiliary parameters
$$
\gamma^\pm:=\delta^\pm-\sum_{i\ge 1} (\alpha_i^\pm+\beta_i^\pm)\ge 0.
$$

To any pair $(\om^+,\om^-)$ we assign a function on the unit circle $\{u\in\C:\,|u|=1\}$
by
$$
\Phi(u;\omega^+,\om^-)=e^{\ga^+(u-1)+\ga^-(u^{-1}-1)}\prod_{i\ge 1} \frac{1+\be_i^+(u-1)}{1-\al_i^+(u-1)}\,
\frac{1+\be_i^-(u^{-1}-1)}{1-\al_i^-(u^{-1}-1)}\,.
$$
This function is analytic in an open neighborhood of the unit circle, where it can be
written as a Laurent series
$$
\Phi(u;\om^+,\om^-)=\sum_{n=-\infty}^\infty \varphi_n(\om^+,\om^-)u^n.
$$

For $\mu\in\GT_N$ set
\begin{equation}\label{eq:add3}
\LGT^\infty_N(\om^+,\om^-;\mu)=\Dim[\mu,N]\cdot \det\bigl[\varphi_{\mu_i-i+j}(\om^+,\om^-)\bigr]_{i,j=1}^N.
\end{equation}

\begin{theorem}\label{th:dGT}
The boundary $\partial \GT$ of the Gelfand-Tsetlin graph can be identified
with the subset in $\wt\Om\times\wt\Om$ determined by the condition $\beta_1^++\beta_1^-\le 1$,
with links $\LGT^\infty_N:\partial \GT\to\GT_N$ given by \eqref{eq:add3}.
\end{theorem}

For the history and various proofs of this deep result (which we propose to
call the {\it Edrei--Voiculescu theorem\/}), see Borodin and Olshanski
\cite{BO11a}, Boyer \cite{Boy83}, Edrei \cite{Ed53}, Okounkov and Olshanski
\cite{OO98},  Vershik and Kerov \cite{VK82}.

Note that $\partial \GT$ is a closed subset in $\wt\Om\times\wt\Om$, thus it is a locally
compact space.

If one replaces the condition $\be_1^++\be_1^-\le 1$ by the weaker one of
$\be_1^{\pm}\le 1$ then \eqref{eq:add3} would still define boundary points, but
each boundary point would correspond to multiple pairs $(\om^+,\om^-)$.

\subsection{The subgraph $\GT^+\subset\GT$}
A signature $\mu\in\GT_N$ is said to be {\it nonnegative\/} if all its
coordinates $\mu_1,\dots,\mu_N$ are nonnegative. Of course, it suffices to
require $\mu_N\ge0$. The nonnegative signatures span a subgraph
$\GT^+=\bigsqcup_{N\ge1}\GT^+_N$ in $\GT$. In what follows we will be concerned
exclusively with this subgraph.

Note that a nonnegative signature may be viewed as a Young diagram. More
precisely, given a Young diagram $\mu\in \Y$ and a positive integer $N$, the
signature $[\mu,N]$ is well defined if and only if $\ell(\mu)$, the number of
nonzero rows in $\mu$, does not exceed $N$.

Let $\mu$ and $\nu$ be two Young diagrams with $\ell(\mu)\le N$ and
$\ell(\nu)\le N+1$, so that vertices $[\mu,N]$ and $[\nu,N+1]$ in $\GT^+$ are
well defined. Then these vertices are joined by an edge, that is,
$[\mu,N]\prec[\nu,N+1]$ if and only if $\mu\subseteq\nu$ and the skew diagram
$\nu/\mu$ is a {\it horizontal strip\/}, meaning that $\nu/\mu$ has at most one
box in each column.

This implies, in particular, that if $[\mu, N]$ and $[\nu,N']$ are in $\GT^+$
and $N'>N$, then $\Dim([\mu,N],[\nu,N'])$ vanishes unless $\mu\subseteq\nu$.

Given Theorem \ref{th:dGT}, it is not hard to see that the boundary $\partial
\GT^+$ can be identified with the subset of $\partial\GT$ determined by
$\om^-=(0,0,0)$.

\subsection{Degeneration $\GT^+\to\YB$}
The next theorem says that the projective system corresponding to the Young
bouquet $\YB$ can be obtained from the projective system corresponding to the
Gelfand--Tsetlin graph (or rather its part $\GT^+$) via a scaling limit
transition turning the discrete scale of levels numbered by $1,2,\dots$ into a
continuous one parametrized by $\R_{>0}$.

Note that the links $\LYB^{r'}_r$ depend on parameters $r'>r$ only through
their ratio $r'/r$.

\begin{theorem}\label{thm5}
Fix arbitrary positive numbers $r'>r>0$ and arbitrary two Young diagrams $\mu$
and $\nu$ such that $\mu\subseteq\nu$. Let two positive integers $N'>N$ go to
infinity in such a way that $N'/N\to r'/r$. Then
\begin{equation}\label{eq19}
\lim\LGT^{N'}_N([\nu,N'],[\mu,N])=\LYB^{r'}_r(\nu,\mu).
\end{equation}
\end{theorem}

\begin{proof}
The idea is to express all the dimensions entering the left- and right-hand
sides through Schur functions and their specializations.

In what follows the brackets $(\,\cdot\,,\,\cdot\,)$ denote the canonical inner product
in $\Sym$; with respect to this product, the Schur functions form an
orthonormal basis. By $(1^N)$ we denote the $N$-tuple $(1,\dots,1)$. We set
$m=|\mu|$ and $n=|\nu|$. Denote by $S_{\nu/\mu}$ the skew Schur function indexed by
the skew diagram $\nu/\mu$, see Section I.5 in \cite{Ma95}.

Here are the necessary formulas:
\begin{gather}
\Dim[\mu,N]=S_\mu(1^N), \qquad \Dim([\mu,N],[\nu,N'])=S_{\nu/\mu}(1^{N'-N})\label{eq16}\\
\dim\mu=(p_1^m,S_\mu), \qquad \dim(\mu,\nu)=(p_1^{n-m},S_{\nu/\mu}).\label{eq17}
\end{gather}

Both in \eqref{eq16} and \eqref{eq17} the first equality is a particular case
of the second one. The first relation in \eqref{eq16} follows from the fact
that the irreducible characters of the unitary groups are given by the Schur
polynomials, and the second equation follows from the combinatorial formula for
the skew Schur functions, see e.g. \cite[I(5.12)]{Ma95}. As for \eqref{eq17},
we first note that $\dim(\mu,\nu)=(S_\mu p_1^{n-m},S_\nu)$ by the simplest
instance of the Pieri rule \cite[I(5.16)]{Ma95}. Then the equality $(S_\mu
p_1^{n-m},S_\nu)=(p_1^{n-m},S_{\nu/\mu})$ follows from \cite[Chapter I,
(5.1)]{Ma95}.

Observe that $p_k(1^N)=N$ for all $k=1,2,\dots$\,. Therefore, if $F$ is a
monomial in $p_1,p_2,\dots$, then $F(1^N)$ equals $N$ raised to the number of letters
in $F$. This number is strictly less than $\deg F$ unless $F$ is a power of
$p_1$. It follows that if $F\in\Sym$ is a homogeneous element, then for large
$N$
$$
F(1^N)=[F: p_1^d]\,N^d+O(N^{d-1}), \qquad d:=\deg F,
$$
where $[F:p_1^d]$ denotes the coefficient of $p_1^d$ in the expansion of $F$ on
monomials in $p_1,p_2,\dots$\,. Next, since the monomials in $p_1,p_2,\dots$
form an orthogonal basis, we have
$$
[F:p_1^d]=\frac{(F,p_1^d)}{(p_1^d,p_1^d)} =\frac{(F,p_1^d)}{d!}
$$
and finally
\begin{equation}\label{eq18}
F(1^N)=\frac{(F,p_1^d)}{d!}\,N^d+O(N^{d-1}), \qquad d:=\deg F.
\end{equation}

Now we proceed to the proof of \eqref{eq19}. By virtue of \eqref{eq23} and
\eqref{eq16}
\begin{align}
\LGT^{N'}_N([\nu,N'],[\mu,N])&=\frac{\Dim[\mu,N]\Dim([\mu,N],[\nu,N'])}{\Dim[\nu,N']}\notag\\
&=\frac{S_\mu(1^N)S_{\nu/\mu}(1^{N'-N})}{S_\nu(1^{N'})}\label{eq24}
\end{align}
Applying \eqref{eq18} to the ordinary and skew Schur functions entering
\eqref{eq24} we get
\begin{align}
S_\mu(1^N)&=\frac{(S_\mu,p_1^m)}{m!}N^m+O(N^{m-1})\notag\\
&=\left(\frac r{r'}\right)^m\frac{(S_\mu,p_1^m)}{m!}(N')^m+O((N')^{m-1})  \label{eq25}\\
S_\nu(1^{N'})&=\frac{(S_\nu,p_1^n)}{n!}(N')^n+O((N')^{n-1})  \label{eq26}\\
S_{\nu/\mu}(1^{N'-N})&=\frac{(S_{\nu/\mu},p_1^{n-m})}{(n-m)!}(N'-N)^{n-m}+O((N'-N)^{n-m-1})\notag\\
&=\left(1-\frac
r{r'}\right)^{n-m}\frac{(S_{\nu/\mu},p_1^{n-m})}{(n-m)!}(N')^{n-m}+O((N')^{n-m-1}).
\label{eq27}
\end{align}

Plugging \eqref{eq25}, \eqref{eq26}, and \eqref{eq27} into \eqref{eq24} we
get
$$
\left(1-\frac r{r'}\right)^{n-m} \left(\frac r{r'}\right)^m
\frac{n!}{(n-m)!\,m!}\,
\frac{(S_\mu,p_1^m)(S_{\nu/\mu},p_1^{n-m})}{(S_\nu,p_1^n)} +O(1/N').
$$
Applying \eqref{eq17} we may rewrite this as
$$
\left(1-\frac r{r'}\right)^{n-m} \left(\frac r{r'}\right)^m
\frac{n!}{(n-m)!\,m!}\, \frac{\dim\mu\,\dim(\mu,\nu)}{\dim\nu} +O(1/N').
$$
Comparing with \eqref{eq29} we see that this is exactly the right-hand side of
\eqref{eq19}, within $O(1/N')$.
\end{proof}

\subsection{Degeneration of the boundary} From Theorem \ref{thm5} it is natural to expect
that there should exist a limit procedure that turns $\partial \GT^+$ into
$\partial\YB$, and our closest goal is to exhibit this procedure.

Each point $\om\in\partial\YB=\wt\Om$ defines a coherent system of measures
$\{{}^\YB M_r\}_{r>0}$ on the levels $\YB_r=\Y$ of $\YB$. Similarly, each point
$(\om^+,\om^-)\in\partial \GT$ defines a coherent system of measures $\{{}^\GT
M_N\}_{N\ge 1}$ on the levels $\GT_N$ of $\GT$. We are about to show that the
former family of coherent systems can be obtained from the latter one by taking
$\om^-=\underline{0}=(0,0,0)$ (since we want to start from $\partial\GT^+$
rather than from $\partial\GT$) and appropriate $\omega^+=\omega^+(\epsilon)$
depending on a small parameter $\epsilon>0$.

As before, we identify nonnegative signatures and Young diagrams.

\begin{theorem}\label{thm7}
 Fix $\om\in\wt\Om$. For any $\mu\in\Y$ and any $r>0$, the following
limiting relation holds: If $N(\epsilon)\sim r\epsilon^{-1}$ as $\epsilon\to +0$
then
$$
\lim_{\epsilon\to +0} \LGT^\infty_{N(\epsilon)} (\epsilon\om,\underline{0};\mu)=
\LYB^\infty_r(\omega,\mu).
$$
\end{theorem}

\begin{proof} In the special case $\omega^-=\underline{0}$, the function
$u\mapsto \Phi(u;\omega^+,\om^-)=\Phi(u;\omega^+,\underline{0})$ is not just
holomorphic in an neighborhood of the unit circle $|u|=1$, but in a
neighborhood of the unit disc $|u|\le 1$. Indeed, all the factors that involve
$\alpha^-_i,\beta^-_i,\gamma^-$ disappear, and the Laurent series turns into a
Taylor series. Thus, all the coefficients $\varphi_n$ with $n<0$ vanish.

This reduces the $N\times N$ determinant in \eqref{eq:add3} to a determinant
of size $\ell=\ell(\mu)$ that does not depend on $N$:
\begin{equation}\label{eq:add4}
\det\bigl[\phi_{\mu_i-i+j}(\om^+,\underline{0})\bigr]_{i,j=1}^N=
\det\bigl[\phi_{\mu_i-i+j}(\om^+,\underline{0})\bigr]_{i,j=1}^\ell \cdot
\bigl(\varphi_0(\omega^+,\underline{0})\bigr)^{N-\ell}.
\end{equation}
This follows from the fact that the $(i,j)$-entry of the matrix in the
left-hand side of \eqref{eq:add3} vanishes for $i>j>\ell$.

Let us now rewrite the expression for $\Phi(u;\omega^+,\underline{0})$ assuming that
$\be_1^+<1$ (this will be automatically satisfied for $\om^+=\epsilon\om$ with small $\epsilon$).
We have
$$
\Phi(u;\omega^+,\underline{0})=e^{\gamma^+(u-1)}\prod_{i=1}^\infty \frac{1+\be^+_i(u-1)}{1-\al^+_i(u-1)}
=e^{-\ga^+}\prod_{i=1}^\infty \frac{1-\be^+_i}{1+\al^+_i}\cdot
 e^{\ga^+u}\prod_{i=1}^\infty \frac{1+\wt\be^+_iu}{1-\wt\al^+_iu}\,,
$$
where
$$
\wt\alpha_i=\frac{\al_i^+}{1+\al_i^+}\,,\quad \wt\be_i^+=\frac{\be_i^+}{1-\be_i^+}\,,\qquad i\ge 1.
$$
Let us substitute $\omega^+=\epsilon\omega$, where
$\om=(\al,\be,\de)\in\wt\Om$. We obtain
$$
\Phi(u;\omega^+,\underline{0}) =e^{-\epsilon\delta}(1+O(\epsilon^2))\cdot
e^{\epsilon\ga u}\prod_{i=1}^\infty
\frac{1+(\epsilon+O(\epsilon^2))\be_iu}{1-(\epsilon+O(\epsilon^2))\al_iu}\,,
$$
where $\ga=\de-\sum_{i\ge1}(\al_i+\be_i)$. All the $O(\epsilon^2)$ terms above
are uniform in $i\ge 1$.

Hence,
$$
\varphi_n(\epsilon\om,\underline{0})=h_n(\om)\epsilon^n+O(\epsilon^{n+1}),\quad n\ge 1,\qquad
\varphi_0(\epsilon\om, \underline{0})=e^{-\epsilon\delta} (1+O(\epsilon^2)).
$$

Taking $N=N(\epsilon)\sim r\epsilon^{-1}$ we see that the determinant \eqref{eq:add4} is
asymptotically equal to
$$
\det\bigl[h_{\mu_i-i+j}(\om)\bigr]_{i,j=1}^\ell \epsilon^{|\mu|}e^{-r\delta} =
S_\mu(\om)\,\epsilon^{|\mu|} e^{-r|\om|},
$$
and, using the hook formula for $\Dim[\mu,N]$ and $\dim\mu$,
\begin{equation}
\label{eq:add5} \Dim[\mu,N]=\frac{\dim\mu}{|\mu|!}\cdot(N)_\mu \sim
\frac{\dim\mu}{|\mu|!}\cdot N^{|\mu|}\sim \frac{\dim\mu}{|\mu|!}r^{|\mu|} \cdot
\epsilon^{-|\mu|}.
\end{equation}
When we multiply these two expressions the factors $\epsilon^{\pm|\mu|}$ cancel
out, and we obtain exactly
$$
\LYB^\infty_r(\om,\mu)=\frac{r^{|\mu|}\dim\mu}{|\mu|!}\,e^{-r|\om|}S_\mu(\om).
$$
\end{proof}

\subsection{ZW-Measures on $\GT$}
Let $\mathcal Z\subset \C^2$ be the disjoint union of the following three sets:
\begin{gather}
\{(z,z')\in\C^2\setminus\R^2\mid z'=\bar z\}, \label{eq36}\\
\{(z,z')\in\R^2\mid \exists m\in\Z, \, m<z,z<m+1\}, \label{eq37}\\
\{(z,z')\in\R^2\mid \exists m\in\Z, \, z=m, \,z'>m-1, \quad \textrm{or} \quad
z'=m,\, z>m-1\}\label{eq35}
\end{gather}
Note that if $(z,z')\in\mathcal Z$, then $z+z'$ is real.

Denote by $\mathcal D_{\operatorname{adm}}$ the subset in $\C^4$ formed by all
quadruples $\zw$ of complex numbers such that:
\begin{itemize}

\item $(z,z')\in\mathcal Z$, $(w,w')\in\mathcal Z$;

\item $z+z'+w+w'>-1$;

\item if both couples $(z,z')$ and $(w,w')$ belong to subsets \eqref{eq35} with
indices $m$ and $\wt m$, respectively, then it is additionally required that
$m+\wt m\ge1$.

\end{itemize}

\begin{definition}
The {\it zw-measure\/} on $\GT_N$ with parameters $\zw\in\mathcal
D_{\operatorname{adm}}$ is given by
$$
\MGT^\zw_N(\mu)= C^\zw_N\cdot\Pi^\zw_N(\mu)\cdot(\Dim[\mu,N])^2
$$
where $\mu$ ranges over $\GT_N$,
$$
C^\zw_N=\prod_{i=1}^N
\frac{\Gamma(z+w+i)\Gamma(z+w'+i)\Gamma(z'+w+i)\Gamma(z'+w'+i)\Gamma(i)}
{\Gamma(z+z'+w+w'+i)}
$$
is a normalization constant, and
\begin{multline*}
\Pi^\zw_N(\mu)\\
=\prod_{i=1}^N \frac1{\Gamma(z-\mu_i+i)\Gamma(z'-\mu_i+i)
\Gamma(w+N+1+\mu_i-i)\Gamma(w'+N+1+\mu_i-i)}.
\end{multline*}

\end{definition}

Note that the measure does not change under transposition $z\leftrightarrow z'$
or $w\leftrightarrow w'$. If both couples $\z$ and $(w,w')$ belong to subset
\eqref{eq36} or subset \eqref{eq37}, that is, none of the four parameters is an
integer, then the expression for $\Pi^\zw_N(\mu)$ is strictly positive for all
$\mu\in\GT_N$. If some of the parameters are integers, then $\Pi^\zw_N(\mu)$
vanishes for some signatures $\mu$. Moreover, if both $\z$ and $(w,w')$ are in
subset \eqref{eq35}, then it may even happen that the normalizing constant has
a singularity, but in such a case the singularity actually disappears after
multiplication by $\Pi^\zw_N(\mu)$ due to cancellation with zeros arising from
appropriate $(1/\Ga(\,\cdot\,))$-factors in $\Pi^\zw_N(\mu)$. Thus, the whole
expression for $\MGT^\zw_N(\mu)$ makes sense for all $\zw\in\mathcal
D_{\operatorname{adm}}$. For more detail, see \cite[Section 7]{Ols03c} and
\cite[Section 3]{BO05a}.

\begin{proposition}
The zw-measures are probability measures, and they are compatible with the
links $\LGT^{N+1}_N:\GT_{N+1}\dasharrow \GT_N$.
\end{proposition}

A proof can be found in \cite[Section 7]{Ols03c}.

Similarly to z-measures that arise in harmonic analysis on the infinite
symmetric group, the zw-measures play a key role in harmonic analysis on the
infinite-dimensional unitary group, see \cite{Ols03c}, \cite{BO05a},
\cite{BO05b}.

\subsection{Degeneration of zw-measures to z-measures}
It is convenient to rewrite the expression for the zw-measures in a slightly
different form:
$$
\MGT^\zw_N(\mu)= \MGT^\zw_N(0^N)\cdot\wt{\Pi}^\zw_N(\mu)\cdot(\Dim[\mu,N])^2
$$
where $0^N=(0,\dots,0)\in\GT_N$ is the zero signature,
$$
\MGT^\zw_N(0^N)=\prod_{i=1}^N\frac{\Gamma(z+w+i)\Gamma(z+w'+i)
\Gamma(z'+w+i)\Gamma(z'+w'+i)\Gamma(i)}{\Gamma(z+z'+w+w'+i)
\Gamma(z+i)\Gamma(z'+i)\Gamma(w+i)\Gamma(w'+i)}\,,
$$
and
\begin{equation}\label{eq1}
\begin{aligned} \wt{\Pi}^\zw_N(\mu)&=\prod_{i=1}^N
\frac{\Gamma(z+i)\Gamma(z'+i)}{\Gamma(z-
\mu_i+i)\Gamma(z'-\mu_i+i)}\\
&\times\prod_{i=1}^N\frac{\Gamma(w+N+1-i)\Gamma(w'+N+1-i)}
{\Gamma(w+N+1+\mu_i-i)\Gamma(w'+N+1+\mu_i-i)}
\end{aligned}
\end{equation}

Assume that $w=0$ while $w'$ is real positive and large enough. Then the
expression for $\MGT^\zw_N(0^N)$ simplifies:
$$
\MGT^\zw_N(0^N)=\prod_{i=1}^N\frac{\Gamma(z+w'+i)
\Gamma(z'+w'+i)}{\Gamma(z+z'+w'+i) \Gamma(w'+i)}\,,
$$
Since $w'$ is assumed to be large, this quantity is nonsingular.

Further, observe that if $\mu_N<0$, then $\wt{\Pi}^\zw_N(\mu)$ vanishes because
$$
\frac{\Ga(w+N+1-i)}{\Ga(w+N+1+\mu_i-i)}\Bigg|_{w=0, \,
i=N}=\frac1{\Ga(1+\mu_N)}=0
$$
and this zero cannot be cancelled after multiplication by $\MGT^\zw_N(0^N)$.
This means that the zw-measure with $w=0$ and $w'$ real and large enough is
concentrated on nonnegative signatures. Thus, we may assume that the measure
lives on the set $\GT^+_N$, which we regard as a subset of $\Y$.

Observe that if $(z,z')$ is admissible in the sense explained in Section
\ref{sc:zmeasures}, parameter $w$ equals $0$, and parameter $w'$ is real and
large enough, then the quadruple $\zw$ belongs to the set $\mathcal
D_{\operatorname{adm}}$ so that the corresponding zw-measure
$\MGT^{(z,z',0,w')}_N$ is well defined for all $N$.

\begin{theorem}\label{thm6}
Fix an arbitrary admissible couple $\z$ of parameters and let parameter $w$
equal\/ $0$ while parameter $w'$ is positive and goes to $+\infty$. Assume that
$N$ varies together with $w'$ in such a way that $N\sim rw'$ with an arbitrary
fixed $r>0$. Then the corresponding zw-measures $\MGT^{(z,z',0,w')}_N$ weakly
converge to the z-measure $\MYB^{(-z,-z')}_r$ on $\Y$.
\end{theorem}

\begin{proof}
It suffices to prove that for any fixed $\mu\in\Y$ and $w'=w'(N)\sim r^{-1}N$,
$$
\lim_{N\to\infty}\MGT^{(z,z',0, w')}_N(\mu)=\MYB^{(-z,-z')}_r(\mu).
$$

{\it Step\/} 1. Let us prove this for $\mu=\varnothing$, which amounts to
$$
\lim_{N\to\infty}\prod_{i=1}^N\frac{\Gamma(z+w'+i)
\Gamma(z'+w'+i)}{\Gamma(z+z'+w'+i) \Gamma(w'+i)}=(1+r)^{-zz'}, \qquad
w'=w'(N)\sim r^{-1}N.
$$

Stirling's formula implies
\begin{gather*}
\frac{\Gamma(z+w'+i)}{\Gamma(w'+i)}=(w'+i)^z\left(1+\frac {z(z-1)}{2(w'+i)}+O\left(N^{-2}\right)\right),\\
\frac{\Gamma(z'+w'+i)}{\Gamma(z+z'+w'+i)}=(z'+w'+i)^{-z}\left(1-\frac
{z(z-1)}{2(z'+w'+i)}+O\left(N^{-2}\right)\right).
\end{gather*}
(Note that, since $w'$ is a large positive number, the arguments of the complex
numbers $z'+w'+i$ are small.) Hence,
$$
\frac{\Gamma(z+w'+i)\Ga(z'+w'+i)}{\Ga(z+z'+w'+i)\Ga(w'+i)}
=\left(\frac{w'+i}{z'+w'+i}\right)^z\left(1+O(N^{-2})\right).
$$
Next, Taylor series type argument shows that
$$
\left(\frac{w'+i}{z'+w'+i}\right)^z=\left(1-\frac{zz'}{w'+i}\right)\left(1+O(N^{-2})\right).
$$
Thus,
$$
\prod_{i=1}^N\frac{\Ga(z+w'+i)\Ga(z'+w'+i)}{\Ga(z+z'+w'+i)\Ga(w'+i)}
=\prod_{i=1}^N\left(1-\frac{zz'}{w'+i}\right)\cdot\left(1+O(N^{-1})\right).
$$
But
\begin{gather*}
\prod_{i=1}^N\left(1-\frac{zz'}{w'+i}\right) =\prod_{i=1}^N\frac{ -zz'+w'+i}{w'+i}
=\frac{\Ga(-zz'+w'+N+1)\Ga(w'+1)}{\Ga(-zz'+w'+1)\Ga(w'+N+1)}\\
=\frac{\Ga(-zz'+w'+N+1)}{\Gamma(w'+N+1)}\,\frac{\Ga(w'+1)}{\Ga(-zz'+w'+1)}
\sim\left(\frac{r^{-1}}{r^{-1}+1}\right)^{zz'}=(1+r)^{-zz'},
\end{gather*}
which gives the desired result.

Note that on the last step we used the well-known asymptotic relation
\begin{equation}\label{eq2}
\frac{\Ga(x+a)}{\Ga(x+b)}\sim x^{a-b}, \qquad \textrm{$x>0$ large}.
\end{equation}

{\it Step\/} 2. It remains to prove that
$$
\lim_{N\to\infty}
\frac{\MGT^{(z,z',0,w')}_N(\mu)}{\MGT^{(z,z',0,w')}_N(\varnothing)}
=\frac{\MYB^{(-z,-z')}_r(\mu)}{\MYB^{(-z,-z')}_r(\varnothing)}, \qquad
w'=w'(N)\sim r^{-1}N,
$$
which amounts to
$$
\lim_{N\to\infty}\left(\wt{\Pi}^{(z,z',0,w')}_N(\mu)\,(\Dim[\mu,N])^2\right)
=(-z)_\mu(-z')_\mu
\left(\frac{r}{1+r}\right)^{|\mu|}\left(\frac{\dim\mu}{|\mu|!}\right)^2
$$

Recall that $\wt{\Pi}^{(z,z',0,w')}_N(\mu)$ is given by formula \eqref{eq1}, which
involves two products over $i=1,\dots,N$. Observe that the $i$th factor in each
of the two products equals 1 when $\mu_i=0$. Since $\mu$ is a Young diagram,
this allows us to restrict each of the products to indices $i=1,\dots,\ell$,
where $\ell$ stands for the number of nonzero rows in $\mu$. Since $\ell$ does
not depend on $N$, we may examine the asymptotics of the factors corresponding
to each $i$ separately.  Using again \eqref{eq2} we get
\begin{align*}
\wt{\Pi}^{(z,z',0,w')}_N(\mu)&\sim\prod_{i=1}^\ell
\frac{\Gamma(z+i)\Gamma(z'+i)}{\Gamma(z- \mu_i+i)\Gamma(z'-\mu_i+i)} \cdot
\left(\frac1{N^2(r^{-1}+1)}\right)^{|\mu|}\\
&=(-z)_\mu(-z')_\mu\left(\frac{r}{1+r}\right)^{|\mu|}N^{-2|\mu|}.
\end{align*}

Further, \eqref{eq:add5} gives
$$
(\Dim[\mu,N])^2\sim\left(\frac{\dim\mu}{|\mu|!}\right)^2\cdot N^{2|\mu|}.
$$

Therefore,
\begin{align*}
\lim_{N\to\infty}\frac{\MGT^{(z,z',0,w')}_N(\mu)}{\MGT^{(z,z',0,w')}_N(\varnothing)}
&= (-z)_\mu(-z')_\mu\left(\frac{r}{1+r}\right)^{|\mu|}
\left(\frac{\dim\mu}{|\mu|!}\right)^2\\
&=\frac{\MYB^{(-z,-z')}_r(\mu)}{\MYB^{(-z,-z')}_r(\varnothing)}.
\end{align*}
\end{proof}

\section{Gibbs measures on the path space}\label{sect5}

\subsection{Gibbs measures}\label{sect5.1}
Let $\Ga$ be a graded graph (Definition \ref{def2}). By a (monotone) {\it
path\/} in $\Ga$ we mean a finite or infinite collection
$$
v_1, e_{12}, v_2, e_{23}, v_3, \dots
$$
where $v_1, v_2, \dots$ are vertices of $\Ga$ such that $|v_{i+1}|=|v_i|+1$ and
$e_{i,i+1}$ is an edge between $v_i$ and $v_{i+1}$. Since we do not consider
more general paths, the adjective ``monotone'' will be omitted. If the graph
has no multiple edges, then every path is uniquely determined by its vertices,
but when multiple edges occur it is necessary to specify which of the edges
between every two consecutive vertices is selected.

Unless otherwise stated, we will assume that the paths start at the lowest
level of the graph. Then the {\it path space\/} $\TT=\TT(\Ga)$ is defined as
the set of all infinite paths.

A {\it cylinder set\/} in $\TT$ is the subset of infinite paths with a
prescribed initial part of finite length. We equip $\TT$ with the Borel
structure generated by the cylinder sets.

\begin{definition}
Let $P$ be a probability measure $P$ on $\TT$. Let us call $P$ a {\it Gibbs
measure\/} if any two initial finite paths with the same endpoint are
equiprobable. Equivalently, the measure of any cylinder set depends only on the
endpoint of the initial part that defines the set.
\end{definition}

This kind of measures on the path space was introduced by Vershik and Kerov
\cite{VK81} under the name of {\it central measures\/}.

As above, consider the projective chain $\{V_N,\La^{N+1}_N\}$ associated with
the graph $\Ga$, so that $V_N$ is the set of vertices of level
$N=1,2,\dots\,$.

\begin{proposition}
There is a natural bijective correspondence between the Gibbs measures on the
path space and coherent systems of measures
$$
\{M_N\}_{N\ge 1}\in\M_\infty=\varprojlim\M(V_N).
$$
\end{proposition}

\begin{proof}
Indeed, given a Gibbs measure $P$, define for each $N$ a probability measure
$M_N\in\M(V_N)$ as follows: For any $v\in V_N$, $M_N(v)$ equals the probability
that the infinite random path distributed according to $P$ passes through $v$.
The measures $M_N$ are compatible with the links $\La^{N+1}_N$ by the very
construction of these links. Therefore, the sequence $(M_N)$ determines an
element of $\M_\infty$. The inverse map, from $\M_\infty$ to Gibbs measures, is
obtained by making use of Kolmogorov's extension theorem.
\end{proof}

Together with Theorem \ref{thm1} this implies

\begin{corollary}\label{cor:gibbs}
There is a bijection between the Gibbs measures on the path space of\/ $\Ga$
and the probability measures on the boundary $\pd\Ga$.
\end{corollary}

Note that the random paths distributed according to a Gibbs measure can be
viewed as trajectories of a Markov chain with discrete time $N$ that flows
backwards from $+\infty$ to $0$ and transition probabilities $\Lambda^{N+1}_N$.
Then $\partial\Gamma$ plays the role of the entrance boundary, and probability
measures on $\partial \Gamma$ turn into entrance laws for the Markov chain.

\subsection{Examples of path spaces for graded graphs}\label{sect5.2}

(a) For the Pascal graph $\Ga=\Pas$, the path space can be identified with the
space $\{0,1\}^\infty$ of infinite binary sequences. Under this identification,
the Gibbs measures are just the exchangeable measures on $\{0,1\}^\infty$, and
the claim of Corollary \ref{cor:gibbs} turns into the classical de Finetti
theorem: exchangeable probability measures on $\{0,1\}^\infty$ are parametrized
by probability measures on $[0,1]$.

(b) Consider the Young graph $\Ga=\Y$. Recall that for a Young diagram
$\la\in\Y$, a standard Young tableau of shape $\la$ is a filling of the boxes
of $\la$ by numbers $1,2,\dots,|\la|$ in such a way that the numbers increase
along each row from left to right and along each column from top to bottom.

Let us also define an {\it infinite Young diagram\/} as an infinite subset
$\wt\la\subseteq\mathbb N\times\mathbb N$ (where $\mathbb N:=\{1,2,\dots\}$)
such that if $(i,j)\in\wt\la$, then $\wt\la$ contains all pairs $(i',j')$ with
$i'\le i$, $j'\le j$.  An {\it infinite standard tableau\/} of shape $\wt\la$
is an assignment of a positive integer to any pair $(i,j)\in\wt\la$ in a such a
way that all positive integers are used, and they increase in both $i$ and $j$.
If we only pay attention to where the integers $1,2,\dots,n$ are located, we
will observe a Young tableau whose shape is a Young diagram $\la\subset\wt\la$
with $n$ boxes. Let us call this finite tableau the {\it $n$-truncation\/} of
the original infinite one.

Clearly, the infinite paths in the Young graph are in one-to-one correspondence
with the infinite Young tableaux. The initial finite parts of such a path are
described by the various trancations of the corresponding tableau. The
condition of a measure on infinite Young tableaux being Gibbs consists in the
requirement that the probability of observing a prescribed truncation depends
only on the shape of the truncation (and not on its filling).

(c) Let us proceed to the Gelfand-Tsetlin graph $\Ga=\GT$. By definition, an
infinite path in $\GT$ is a sequence $\la^{(1)}\prec\la^{(2)}\prec\dots$ with
$\la^{(N)}\in\GT_N$. If one defines $x_{i}^j=\lambda^{(j)}_i$ then one easily
sees that such paths are in one-to-one correspondence with the infinite
triangular arrays $\{x_i^j\mid 1\le i\le j,\,j\ge 1\}$ of integers that satisfy
the inequalities
$$
x^{j+1}_i\ge x_i^j\ge x_{i+1}^{j+1}
$$
for all meaningful indices $(i,j)$. Such arrays are called infinite {\it
Gelfand-Tsetlin schemes\/}. The initial finite parts of infinite paths in a
similar way give rise to finite Gelfand-Tsetlin schemes. Infinite
Gelfand-Tsetlin schemes are also in one-to-one correspondence to certain
tilings of a half-plane by lozenges, see the introduction to \cite{BF08}.

If we restrict ourselves to the subgraph $\GT^+\subset\GT$,  then the
signatures can be identified with Young diagrams, and infinite paths may be
viewed as infinite {\it semi-standard\/} Young tableaux, where
``semi-standard'' refers to the condition that the filling numbers are only
required to {\it weakly\/} increase along rows, and they are also not required
to exhaust all positive integers. The finite paths of length $N$ then turn into
semi-standard Young tableaux whose shape has no more than $N$ rows.

(d) Other examples of Gibbs measures on path spaces related to exchangeability
can be found in Kingman \cite{Ki78}  (exchangeable partitions of $\mathbb N$),
Gnedin \cite{Gn97} (exchangeable ordered partitions of $\mathbb N$), Gnedin and
Olshanski \cite{GO06} (exchangeable orderings of $\mathbb N$).

\subsection{Path spaces for $\mathbb B$ and $\YB$}
Similarly to the case of graded graphs described above, one can define Gibbs
measures on paths corresponding to more general projective systems. Without
going into general definitions, let us describe the outcome in the cases of the
binomial system $\mathbb B$ and the Young bouquet $\YB$.

Recall that the levels of the binomial system $\mathbb B$ are labelled by
numbers $r\in\R_{>0}$ (strictly positive real numbers), and each level consists
of points $m\in \Z_+:=\{0,1,2,\dots\}$. It is convenient to denote these points
as pairs $(m,r)\in \Z_+\times \R_{>0}$, and also add the point $(0,0)$ at level
0.

An infinite path in $\mathbb B$ can be viewed as an integer-valued function
$m=m(r)$, $m(0)=0$, that is weakly
increasing, left-continuous, and has only jumps of size 1: $m(r+0)-m(r)\in\{0,1\}$
for any $r>0$.

The jump locations for $m(r)$ form a (possibly empty)  increasing sequence
$r_1<r_2<\dots$ tending to $+\infty$. Thus, a path in $\mathbb B$ may be
encoded by a locally finite point configuration in the space $\R_+$ of
nonnegative real numbers.

A probability measure on the infinite paths in $\mathbb B$ (equivalently, point
configurations in $\R_+$) is {\it Gibbs\/} if for any $n\ge 0$, under the
condition that a segment $[0,r]\subset \R_+$ contains exactly $n$ jumps at
$r_1,\dots,r_n$, the distribution of their locations is proportional to the
Lebesgue measure $dr_1\cdots dr_n$.

One shows that coherent systems on $\mathbb B$ are in one-to-one correspondence
with the Gibbs measures as defined above.

The extreme Gibbs measure corresponding to a point $x\in\R_+=\partial\mathbb
B$, $x\ne 0$, corresponds to the Poisson process on $\R_+$ with constant
intensity $x$. The extreme Gibbs measure corresponding to $x=0$ is the
delta-measure on the path $m(r)\equiv 0$.

A general Gibbs measure is thus a (possibly continuous) convex combination of
the delta-measure at the zero path and a random mix of the Poisson processes
with constant intensities, also known as a doubly stochastic Poisson process,
or a Cox process.

Let us proceed to $\YB$. The construction is a combination of those for $\Y$
and for $\mathbb B$.

Recall that an element of $\YB$ is a pair $(\la,r)\in\Y\times \R_+$ with the
condition that $\la=\varnothing$ if $r=0$. A path in $\YB$ is defined as a
monotonically increasing Young diagram-valued function $\la(r)$,
$\la(r')\supseteq\la(r)$ for $r'>r$, such that $(|\lambda(r)|,r)$ is a path in
$\mathbb B$.

Such a path can be encoded by a {\it generalized\/} standard Young tableau,
whose shape is a finite or infinite Young diagram and filling numbers are
positive reals (strictly increasing along rows and columns) that have no finite
accumulation points.

A finite initial part of a path is then given by the following data: a real
number $r>0$, an integer $n\ge 0$, a collection of $n$ numbers
$0<r_1<\dots<r_n\le r$, and a standard Young tableau whose shape has $n$ boxes.
The Gibbs property consists in requiring that the distribution of the
coordinates $r_1,\dots,r_n$ is proportional to the Lebesgue measure $dr_1\cdots
dr_n$ on the polytope in $\R_+^n$ cut out by the inequalities that guarantee
row and column monotonicity of the coordinates.

Once again, the Gibbs measures are in one-to-one correspondence with the
probability measures on $\wt\Omega=\partial \YB$.

Every probability measure $M$ on the boundary $\pd\B$ or $\pd\YB$ serves as the
entrance law of a Markov process on $\Z_+$ or $\Y$, respectively, with ``time''
$r$ ranging from $+\infty$ to $0$ (a more conventional picture is obtained by
taking as time $t:=-\log r$). The trajectories of this process are the paths as
described above, and the Gibbs measure corresponding to $M$ is the law of the
process.

\subsection{Path degeneration $\GT^+\to\YB$}
To conclude, let us see how the degeneration $\GT^+\to \YB$ described in
Theorem \ref{thm1} works on the level of Gibbs measures on paths. Consider all
finite paths in $\GT^+$ that have a given nonnegative signature $[\lambda,N]$
as their final point. They may be viewed as semi-standard Young tableaux of
shape $\la$ filled with (some of the) numbers $1,\dots,N$. By definition of the
Gibbs property, all those tableaux must have equal probabilities for any Gibbs
measure on the path space of $\GT^+$.

Let us further consider the asymptotics when $\la$ stays fixed and $N=rL$ with
a fixed $r>0$ and $L\gg 1$. Then if we take the random path in $\GT^+$ that
ends at $[\lambda,N]$ and divide the entries in the corresponding Young tableau
by $L$, we will observe a random generalized Young tableau of shape $\la$ with
filling numbers not exceeding $r$, or a finite path in $\YB$ ending at
$(\la,r)$. Its asymptotic distribution will be proportional to the Lebesgue
measure on the polytope of the filling numbers, and this is exactly what is
required by the Gibbs property on $\YB$.


\begin{thebibliography}{KOV04}

\bibitem[AESW51]{AESW51}
M. Aissen, A. Edrei, I. J. Schoenberg, and A. Whitney, {\it On the generating
functions of totally positive sequences\/}. Proc. Nat. Acad. Sci. USA {\bf37}
(1951), 303--307.

\bibitem[ASW52]{ASW52}
M. Aissen, I. J. Schoenberg, and A. Whitney, {\it On the generating functions
of totally positive sequences I\/}. J. Analyse Math. {\bf2} (1952), 93--103.

\bibitem[BDJ99]{BDJ99}
J. Baik, P. Deift, and K. Johansson, {\it On the distribution of the length of
the longest increasing subsequence of random permutations\/}. J. Amer. Math.
Soc. {\bf12} (1999), 1119--1178 [arXiv:math/9810105].

\bibitem[BF08]{BF08}
A.~Borodin and P.~L.~Ferrari, {\it Anisotropic growth of random surfaces in 2+1
dimensions\/}. arXiv:0804.3035.

\bibitem[BO00a]{BO00a}
A.~Borodin and G.~Olshanski, {\it Distributions on partitions, point processes,
and the hypergeometric kernel\/}. Commun. Math. Phys. {\bf211} (2000), no. 2,
335--358;  arXiv:math/9904010.

\bibitem[BO00b]{BO00b}
A.~Borodin and G.~Olshanski, {\it Harmonic functions on multiplicative graphs
and interpolation polynomials\/}. Electronic J. Comb. {\bf7} (2000), Paper
\#R28; math/9912124.

\bibitem[BO05a]{BO05a}
A.~Borodin and G.~Olshanski, {\it  Harmonic analysis on the
infinite-dimensional unitary group and determinantal point processes\/}. Ann.
of Math. (2) {\bf161} (2005), no. 3, 1319--1422.


\bibitem[BO05b]{BO05b}
A.~Borodin and G.~Olshanski, {\it Representation theory and random point
processes\/}. In: A.~Lap\-tev (ed.), {\it European congress of mathematics\/}
(ECM), Stockholm, Sweden, June 27--July 2, 2004. Z\"urich: European
Mathematical Society, 2005, pp. 73--94.

\bibitem[BO06]{BO06}
A.~Borodin and G.~Olshanski, {\it Markov processes on partitions\/}. Probab.
Theory  Rel. Fields {\bf135} (2006), 84--152; arXiv:math-ph/0409075.

\bibitem[BO09]{BO09}
A.~Borodin and G.~Olshanski,  {\it Infinite-dimensional diffusions as limits of
random walks on partitions\/}. Probab. Theory Rel. Fields {\bf144} (2009),
281--318; arXiv:0706.1034.

\bibitem[BO10]{BO10}
A.~Borodin and G.~Olshanski, {\it Markov processes on the path space of the
Gelfand-Tsetlin graph and on its boundary\/}, arXiv:1009.2029 .

\bibitem[BO11a]{BO11a}
A.~Borodin and G.~Olshanski, {\it The boundary of the Gelfand--Tsetlin graph: A
new approach\/}. Preprint arXiv:1109.1412.

\bibitem[BO11b]{BO11b}
A.~Borodin and G.~Olshanski, paper in preparation.

\bibitem[Boy83]{Boy83} R.~P.~Boyer, {\it
Infinite traces of AF-algebras and characters of $U(\infty)$\/}. J.\ Operator
Theory {\bf 9} (1983), 205--236.

\bibitem[Br72]{Br72}
O. Bratteli, {\it Inductive limits of finite-dimensional $C^*$-algebras\/}.
Trans. Amer. Math. Soc. {bf171} (1972), 195--234.

\bibitem[Cur99]{Cur99}
C. W. Curtis, {\it Pioneers of representation theory: Frobenius, Burnside,
Schur, and Brauer\/}.  Amer. Math. Soc., Providence, RI, 1999.


\bibitem[Dy71]{Dy71}
E. B. Dynkin, {\it Initial and final behavior of trajectories of Markov
processes\/}. Uspehi Mat. Nauk {\bf26} (1971), no. 4, 153--172 (Russian);
English translation: Russian Math. Surveys {\bf26} (1971), no. 4, 165--185.


\bibitem[Dy78]{Dy78}
E. B. Dynkin, {\it Sufficient statistics and extreme points\/}. Ann. Probab.
{\bf6} (1978), 705--730.

\bibitem[Ed52]{Ed52}
A. Edrei, {\it On the generating functions of totally positive sequences II\/}.
J. Analyse Math., {\bf2} (1952), 104--109.

\bibitem[Ed53]{Ed53}
A. Edrei, {\it On the generating function of a doubly infinite, totally
positive sequence\/}. Trans. Amer. Math. Soc. {\bf74} (1953), 367--383.

\bibitem[Gn97]{Gn97}
A. V. Gnedin, {\it The representation of composition structures\/}. Ann. Prob.
{\bf25} (1997), 1437--1450.

\bibitem[GO06]{GO06}
A. Gnedin and G. Olshanski, {\it Coherent permutations with descent statistic
and the boundary problem for the graph of zigzag diagrams\/}. Internat. Math.
Research Notices {\bf2006} (2006), Article ID 51968.

\bibitem[Gor10]{Gor10}
V. Gorin, {\it Disjointness of representations arising in the problem of
harmonic analysis on an infinite-dimensional unitary group\/}. Funct. Anal.
Appl.  {\bf44}  (2010), no. 2,  92--105; arXiv:0805.2660.

\bibitem[Ke03]{Ke03}
S. V. Kerov, {\it Asymptotic representation theory of the symmetric group and
its applications in analysis\/}. Translations of Mathematical Monographs,
{\bf219}. American Mathematical Society, Providence, RI, 2003.

\bibitem[KOO98]{KOO98}
S. Kerov, A. Okounkov, and G. Olshanski, {\it The boundary of Young graph with
Jack edge multiplicities\/}. Intern. Mathematics Research Notices, 1998, no. 4,
173--199; arXiv:q-alg/9703037.


\bibitem[KOV93]{KOV93}
S.~Kerov, G.~Olshanski, and A.~Vershik, {\it Harmonic analysis on the infinite
symmetric group. A deformation of the regular representation\/}. Comptes Rendus
Acad. Sci. Paris, S\'er. I {\bf316} (1993), 773--778.

\bibitem[KOV04]{KOV04}
S.~Kerov, G.~Olshanski, and A.~Vershik, {\it Harmonic analysis on the infinite
symmetric group\/}. Invent. Math. {\bf158}  (2004), 551--642;
arXiv:math/0312270.

\bibitem[KeOr90]{KeOr90}
S. V. Kerov, O. A. Orevkova, {\it Random processes with common cotransition
probabilities\/}.  Zapiski Nauchnyh Seminarov LOMI {\bf 184} (1990), 169--181
(Russian); English translation in J. Math. Sci. (New York) {\bf68} (1994), no.
4, 516--525.

\bibitem[Ki78]{Ki78}
J. F. C. Kingman,  {\it The representation of partition structures\/}. J.
London Math. Soc. {\bf18} (1978), 374--380.

\bibitem[Ma95]{Ma95} I.~G.~Macdonald, {\it Symmetric Functions and Hall Polynomials\/}.
Clarendon Press Oxford, 1995.

\bibitem[Mack57]{Mack57}
G. W. Mackey, {\it Borel structure in groups and their duals\/}. Trans. Amer.
Math. Soc. {\bf85} (1957), 134--165.

\bibitem[Mey66]{Mey66}
P.-A. Meyer, {\it Probability and potentials\/}. Blaisdell, 1966.

\bibitem[Ner02]{Ner02}
Yu. A. Neretin, {\it Hua-type integrals over unitary groups and over projective
limits of unitary groups\/}.  Duke Math. J. {\bf114} (2002), 239--266;
arXiv:math-ph/0010014.

\bibitem[OO97]{OO97}
A. Okounkov and G. Olshanski, {\it Shifted Schur functions\/}. Algebra i Analiz
{\bf9} (1997), no. 2, 73--146 (Russian); English version: St. Petersburg
Mathematical J., 9 (1998), 239--300; arXiv:q-alg/9605042.

\bibitem[OO98]{OO98}
A.~Okounkov and G.~Olshanski, {\it Asymptotics of Jack polynomials as the
number of variables goes to infinity\/}. Intern. Math. Res. Notices 1998, no.
13, 641--682.

\bibitem[Ols03a]{Ols03a}
G. Olshanski, {\it  Point processes related to the infinite symmetric group\/}.
In: The orbit method in geometry and physics: in honor of A. A. Kirillov (Ch.
Duval et al., eds.), Progress in Mathematics {\bf213}, Birkh\"auser, 2003, pp.
349–393;  arXiv:math.RT/9804086.

\bibitem[Ols03b]{Ols03b}
G. Olshanski, {\it An introduction to harmonic analysis on the infinite
symmetric group\/}. In: Asymptotic Combinatorics with Applications to
Mathematical Physics (A.~Vershik, ed.). Springer Lecture Notes in Math. {\bf
1815}, 2003, 127--160.

\bibitem[Ols03c]{Ols03c}
G. Olshanski, {\it The problem of harmonic analysis on the infinite-dimensional
unitary group\/}. J. Funct. Anal. {\bf205} (2003), no. 2, 464--524.

\bibitem[Ols10]{Ols10}
G. Olshanski, {\it Laguerre and Meixner symmetric functions, and
infinite-dimensional diffusion processes\/}. Zapiski Nauchnyh Seminarov POMI
{\bf378} (2010), 81--110. Reproduced in J. Math. Sci. (New York)  {\bf174}
(2011), No. 1, 41--57; arXiv:1009.2037.

\bibitem[Ols11]{Ols11}
G. Olshanski, {\it Laguerre and Meixner  orthogonal bases in the algebra of
symmetric functions\/}. Intern. Math. Research Notices, to appear;
arXiv:1103.5848.

\bibitem[ORV03]{ORV03}
G. Olshanski, A. Regev, and A. Vershik, {\it Frobenius-Schur functions\/}. In:
Studies in memory of Issai Schur (A.~Joseph, A.~Melnikov, R.~Rentschler, eds.).
Progress in Mathematics  {\bf210}, pp. 251--300. Birkh\"auser, 2003;
arXiv:math/0110077.

\bibitem[Osi11]{Osi11}
A. Osinenko, {\it Harmonic analysis on the infinite-dimensional unitary
group\/}. To appear.

\bibitem[Pa67]{Pa67}
K. R. Parthasarathy, {\it Probability measures on metric spaces\/}. Academic
Press, Inc., New York-London 1967.

\bibitem[Pic87]{Pic87}
D.~Pickrell, {\it Measures on infinite dimensional Grassmann manifold\/}.
J.~Func.\ Anal.\ {\bf70} (1987); 323--356.

\bibitem[Sa01]{Sa01}
B. E. Sagan, {\it The symmetric group. Representations, combinatorial
algorithms, and symmetric functions\/}. Second edition, Springer, 2001.

\bibitem[Tho64]{Tho64}
E. Thoma, {\it Die unzerlegbaren, positive--definiten Klassenfunktionen der
abz\"ahlbar unendlichen, symmetrischen Gruppe\/}. Math. Zeitschr., {\bf85}
(1964), 40--61.

\bibitem[VK81]{VK81}
A. M. Vershik and S. V. Kerov, {\it Asymptotic theory of characters of the
symmetric group\/}. Funct. Anal. Appl. {\bf15} (1981), 246--255.

\bibitem[VK82]{VK82}
A. M. Vershik and S. V. Kerov, {\it Characters and factor representations of
the infinite unitary group\/}. Doklady AN SSSR {\bf267} (1982), no. 2, 272--276
(Russian); English translation: Soviet Math. Doklady {\bf26} (1982), 570--574.

\bibitem[VK90]{VK90}
A. M. Vershik and S. V. Kerov, {\it The Grothendieck group of infinite
symmetric group and symmetric functions (with the elements of the theory of
$K_0$-functor for AF-algebas)\/}. In: Representations of Lie groups and related
topics. Advances in Contemp. Math., vol. 7 (A.~M.~Vershik and D.~P.~Zhelobenko,
editors). Gordon and Breach, N.Y., London etc. 1990,  39--117.

\bibitem[Vo76]{Vo76}
D.~Voiculescu, {\it Repr\'esentations factorielles de type {\rm II}${}_1$ de
$U(\infty)$\/}. J. Math. Pures et Appl. {\bf55} (1976), 1--20.

\bibitem[Wey39]{Wey39}
H. Weyl, {\it The classical groups. Their invariants and representations\/}.
Princeton Univ. Press, 1939; 1997 (fifth edition).

\bibitem[Wi85]{Wi85}
G. Winkler, {\it Choquet order and simplices. With applications in
probabilistic models\/}. Springer Lect. Notes Math. {\bf1145}, 1985.

\bibitem[Zhe70]{Zhe70}
D. P. Zhelobenko, {\it Compact Lie groups and their representations\/}, Nauka,
Moscow, 1970 (Russian); English translation: Transl. Math. Monographs 40, Amer.
Math. Soc., Providence, RI, 1973.

\end{thebibliography}
\end{document}